\documentclass[12pt, reqno]{amsart}
\usepackage[utf8]{inputenc}
\usepackage[T2A]{fontenc}
\usepackage[english]{babel}
\usepackage{amsmath,amsfonts,amssymb}

\usepackage[margin = 2.45cm]{geometry}
\usepackage[inline]{enumitem}
\usepackage[svgcolors, dvipsnames]{xcolor}
\usepackage[all]{xy}
\usepackage{amssymb, amsmath, amscd, amsthm}
\usepackage[colorlinks=true]{hyperref}
\hypersetup{urlcolor=blue, linkcolor=blue, citecolor=blue}

\newtheorem{theorem}[subsection]{Theorem}
\newtheorem{lemma}[subsection]{Lemma}
\newtheorem{sublemma}[subsubsection]{Lemma}

\newtheorem{corollary}[subsection]{Corollary}

\newtheorem{definition}[subsection]{Definition}
\newtheorem{remark}[subsection]{Remark}
\newtheorem{example}[subsection]{Example}

\makeatletter
\@addtoreset{subsection}{section}
\@addtoreset{equation}{section}
\@addtoreset{figure}{section}
\@addtoreset{table}{section}
\makeatother

\makeatletter
\newcommand\testshape{family=\f@family; series=\f@series; shape=\f@shape.}
\def\myemphInternal#1{\if n\f@shape%
\begingroup\itshape #1\endgroup\/%
\else\begingroup\sf\itshape #1\endgroup%
\fi}
\def\myemph{\futurelet\testchar\MaybeOptArgmyemph}
\def\MaybeOptArgmyemph{\ifx[\testchar \let\next\OptArgmyemph
                 \else \let\next\NoOptArgmyemph \fi \next}
\def\OptArgmyemph[#1]#2{\index{#1}\myemphInternal{#2}}
\def\NoOptArgmyemph#1{\myemphInternal{#1}}
\makeatother
\newcommand\id{\mathrm{id}}          % identity map    
\newcommand\Int{\mathrm{Int}}        % interior 
\newcommand\eps{\varepsilon}

\newcommand\Diff{\mathcal{D}}       % diffeomorphisms
\newcommand\Map{\mathrm{Map}}       % maps
\newcommand\bR{\mathbb{R}}

\newcommand\Cr[1]{\mathcal{C}^{#1}}
\newcommand\Cinfty{\mathcal{C}^{\infty}}

\newcommand\Crm[3]{\Cr{#1}\!\left(#2,#3\right)}
\newcommand\Crmx[3]{\Cr{#1}\bigl(#2,#3\bigr)}
\newcommand\Cont[2]{\Crm{0}{#1}{#2}}                         % space of continuous maps

\newcommand\Ci[2]{\mathcal{C}^{\infty}(#1,#2)}               % space of C^\infty maps
\newcommand\Aman{A}
\newcommand\Bman{B}
\newcommand\Cman{C}
\newcommand\Dman{D}
\newcommand\Eman{E}

\newcommand\Iman{I}
\newcommand\Jman{J}
\newcommand\Kman{K}
\newcommand\Lman{L}
\newcommand\Mman{M}
\newcommand\Nman{N}

\newcommand\Qman{Q}

\newcommand\Uman{U}
\newcommand\Vman{V}
\newcommand\Wman{W}
\newcommand\Xman{X}
\newcommand\Yman{Y}
\newcommand\Zman{Z}
\newcommand\sAA{\mathcal{A}}
\newcommand\sBB{\mathcal{B}}
\newcommand\NN{\mathcal{N}}
\newcommand\UU{\mathcal{U}}

\newcommand\WW{\mathcal{W}}

\newcommand\AFunc{A}
\newcommand\BFunc{B}
\newcommand\Func{F}
\newcommand\GFunc{G}

\newcommand\func{f}
\newcommand\gfunc{g}

\newcommand\dif{h}

\newcommand\Disk[1]{\Dman^{#1}}

\newcommand\Wtop{\mathsf{W}}
\newcommand\Stop{\mathsf{S}}
\newcommand\Wr[1]{\Wtop^{#1}}
\newcommand\Sr[1]{\Stop^{#1}}

\newcommand\Jr[3]{\mathcal{J}^{#1}(#2,#3)}

\newcommand\jetM[2]{j^{#1}_{#2}}
\newcommand\jr[3]{j_{#3}^{#1}[#2]}
\newcommand\jrm[3]{\jetM{#1}{#2}(#3)}

\newcommand\jpr[1]{q_{#1}}
\newcommand\pr[1]{p_{#1}}
\newcommand\src[1]{\mathbf{src}}
\newcommand\dest[1]{\mathbf{dest}}
\newcommand\ejr[1]{\mathbf{j}^{#1}}

\newcommand\WWW[4]{\Crm{#1}{#3}{#4}^{#2}}

\newcommand\WNBH[3]{[#2,#3]}

\newcommand\indmap[1]{#1_{*}}
\newcommand\indmapx[1]{#1^{*}}
\newcommand\incl{i}

\newcommand\rx{r} %{{\color{Green}r}}
\newcommand\sx{s} %{{\color{Red}s}}

\newcommand\ptt{t} %{{\color{Orange}t}}  % usef for homotopies
\newcommand\pts{z} %{{\color{Red}z}}
\newcommand\ptq{q} %{{\color{Red}q}}
\newcommand\ptl{v} %{{\color{Green}l}}
\newcommand\ptm{x} %{{\color{Mulberry}x}}
\newcommand\ptn{y} %{{\color{Red}y}}

\newcommand\vfunc{v}
\newcommand\bpt{q} %{{\color{Red}\mathbf{q}}}

\newcommand\pfunc{\varphi}

\newcommand\BfX{\sBB_{\pfunc,\Xman}}
\newcommand\AfX{\sAA_{\pfunc,\Xman}}

\newcommand\EXP{\mathbf{E}}

\newcommand\UNbh{\mathcal{U}}
\newcommand\VNbh{\mathcal{V}}

\newcommand\GOOD[5]{$(#2, #3, #5)$-preserving}
\newcommand\fgood{\GOOD{\func}{\sBB}{\Xman}{\rx}{\sx}}

\newcommand\betagood{\GOOD{\BFunc}{\UNbh}{\Disk{k+1}\times\Xman}{\rx}{\sx}}

\newcommand\mult{\mathsf{m}}

\newcommand\Mor[3]{\mathcal{C}_{\times}^{#1}(#2,#3)}

\newcommand\LMiso{\Sigma}
\newcommand\Niso{\Delta}
\newcommand\Hl{\lambda}
\newcommand\Hm{\mu}

\newcommand\DiffIMX[1]{\widehat{\Diff}_{\Jman \times \Mman}^{#1}(\Iman\times \Mman)}
\newcommand\DiffIMdIM{\DiffIMX{\rx}}
\newcommand\DiffsIMdIM{\DiffIMX{\sx}}

\newcommand\LoopDM{\Omega(\Diff^{\rx}(\Mman))}
\newcommand\LoopDMJ{\Omega_{\Jman}(\Diff^{\rx}(\Mman))}

\newcommand\cnorm[3]{\|#1\|_{#2,#3}}
\newcommand\relcnorm[4]{\|#1\|_{#2,#3}^{#4}}

\newcommand\sFunc{B} %{\color{Red}B}}
\newcommand\BXSiso{H} %{\color{PineGreen}H}}

\newcommand\CrmPairs[5]{\mathcal{C}^{#1}_{#3,\,#5}(#2,#4)}
\newcommand\CrMdMMdM[1]{\CrmPairs{#1}{\Mman}{\partial\Mman}{\Mman}{\partial\Mman}}
\newcommand\CrMdMNdN[1]{\CrmPairs{#1}{\Mman}{\partial\Mman}{\Nman}{\partial\Nman}}
\newcommand\CrmLMdMNdN[1]{\CrmPairs{#1}{\Lman\times\Mman}{\Lman\times\partial\Mman}{\Nman}{\partial\Nman}}

\title{Smooth approximations and their applications to homotopy types}
\author{Oleksandra Khokhliuk}
\address{Department of Geometry, Topology and Dynamical Systems \\
Taras Shevchenko National University of Kyiv \\
Hlushkova Avenue, 4e, Kyiv, 03127 Ukraine}
\email{khokhliyk@gmail.com}

\author{Sergiy Maksymenko}
\address{Topology Laboratory, Institute of Mathematics of NAS of Ukraine \\ Tereshchenkivska str., 3, Kyiv, 01024 Ukraine}
\email{maks@imath.kiev.ua}

\begin{document}
\begin{abstract}
Let $M, N$ the be smooth manifolds, $\mathcal{C}^{r}(M,N)$ the space of ${C}^{r}$ maps endowed with weak $C^{r}$ Whitney topology, and $\mathcal{B} \subset \mathcal{C}^{r}(M,N)$ an open subset.
It is proved that for $0\leq r<s\leq\infty$ the inclusion $\mathcal{B} \cap \mathcal{C}^{s}(M,N) \subset \mathcal{B}$ is a weak homotopy equivalence.
It is also established a parametrized variant of such a result.
In particular, it is shown that for a compact manifold $M$, the inclusion of the space of $\mathcal{C}^{s}$ isotopies $[0,1]\times M \to M$ fixed near $\{0,1\}\times M$ into the space of loops $\Omega(\mathcal{D}^{r}(M), \mathrm{id}_{M})$ of the group of $\mathcal{C}^{r}$ diffeomorphisms of $M$ at $\mathrm{id}_{M}$ is a weak homotopy equivalence.
\end{abstract}

\maketitle

% !TeX encoding = UTF-8
% !TeX spellcheck = en_US

\section{Introduction}
Let $\Mman,\Nman$ be two smooth manifolds.
Then for each $r\geq0$ the space $\Crm{\rx}{\Mman}{\Nman}$ of $\Cr{\rx}$ maps $\Mman\to\Nman$ can be endowed with the \myemph{$\Cr{\rx}$ weak Whitney topology} which we will denote by $\Wr{\rx}$, e.g.~\cite[Chapter~2]{Hirsch:DiffTop}, \cite[Chapter~8]{Mukherjee:DT:2015}.
In fact, 
\begin{itemize}[leftmargin=5ex]
\item 
$\Wr{0}$ is the same as the compact open topology, 
\item $\Wr{\rx}$ for $\rx<\infty$ arises as the induced compact open topology from the natural inclusion $\Crm{\rx}{\Mman}{\Nman} \subset \Cont{\Mman}{\Jr{\rx}{\Mman}{\Nman}}$ (called \myemph{$\rx$-jet prolongation}), where $\Jr{\rx}{\Mman}{\Nman}$ is the space of $\rx$-jets of maps $\Mman\to\Nman$, and 
\item $\Wr{\infty}$ is generated by all topologies $\Wr{\rx}$ for $0\leq\rx<\infty$.
\end{itemize}

In particular, we get a sequence of continuous inclusions
\begin{equation}\label{equ:CrMN_inclusions}
\Ci{\Mman}{\Nman}
\ \subset \
\cdots
\ \subset \
\Crm{\rx}{\Mman}{\Nman}
\ \subset \
\Crm{r-1}{\Mman}{\Nman}
\ \subset \
\cdots
\ \subset \
\Cont{\Mman}{\Nman}
\end{equation}
that are not topological embeddings, since each topology $\Wr{\rx}$ is stronger than the induced $\Wr{r-1}$ one.

It seems to be known that each of the above inclusions is a \myemph{weak homotopy equivalence}, but the authors did not succeed to find its explicit formulation in the available literature, though it was proved in~\cite[Theorem~A.3.7 \& Remark~A.3.8]{Neeb:AIF:2002} for the case when $\Mman$ is compact and $\Nman$ is a Lie group (even a Fr\'{e}chet-Lie group), and there are discussions, claims that it is true in general, and ideas of proof in some internet mathematical resources, see e.g.~\cite[answer by Ryan Budney]{FigueroaFarrill:MOF:2010} and \cite{Clough:MathStackExch:2016}.
However, after posting to arXiv the initial version of the present paper Prof.\! Helge Gl\"{o}ckner kindly informed us that such a statement follows from his (currently yet unpublished) paper~\cite[Proposition~6.1]{Glockner:0812.4713:2008}, see Remark~\ref{rem:comparison_with_Glockner} for details.
That weak homotopy equivalence is a consequence of $\Cr{\infty}$ approximation theorems and extends~\cite[Theorems~12, 16, 17]{Palais:Top:1966} and~\cite{Shvarts:DANSSSR:1964} about homotopy types of open subsets of locally convex topological vector spaces.

\begin{remark}\rm
If $\Mman$ is compact, then by~\cite[Corollary~2]{Milnor:TrAMS:1959} (being simply a combination of results by J.H.C.~Whitehead, O.~Hanner, and K.~Kuratowski), the spaces in~\eqref{equ:CrMN_inclusions} have homotopy types of countable CW-complexes, and therefore all the above inclusions are in fact \myemph{homotopy equivalences}.
\end{remark}

The aim of this paper is to give elementary and thorough proofs that several kinds of natural inclusions between spaces of differentiable maps, including inclusions~\eqref{equ:CrMN_inclusions}, are (weak) homotopy equivalences as well, see Theorems~\ref{th:B_CsMN__B}, \ref{th:B_CsMN__B_rel_param}, and~\ref{th:B_CsMN__B_rel_param_pairs}.

The obtained results are thus partially contained in~\cite{Glockner:0812.4713:2008}. 
We work in not such a generality as in~\cite{Glockner:0812.4713:2008}, but we additionally cover 
\begin{itemize}[leftmargin=*]
\item 
the inclusions $\sBB \cap\Crm{\sx}{\Mman}{\Nman} \subset \sBB$, where $\sBB$ is an open subset of $\Crm{\rx}{\Mman}{\Nman}$;
\item 
the ``\myemph{relative}'' case of spaces of maps $\Mman\to\Nman$ that coincide with some map $\pfunc:\Mman\to\Nman$ on a closed subset $\Xman\subset\Mman$, e.g. groups of diffeomorphisms fixed on some tubular neighbourhood of a submanifold, and in particular on a collar of the boundary of a manifold, Corollary~\ref{cor:tubular_nbh_whe_anyM};
\item
the ``\myemph{parametrized}'' case of spaces of continuous maps of a manifold $\Lman$ into the space of $\Cr{\rx}$ maps $\Mman\to\Nman$, and in particular, show that certain spaces of smooth isotopies are weakly homotopy equivalent to the corresponding loop spaces (consisting of continuous maps), Theorem~\ref{th:loop_space} and Corollary~\ref{cor:loop_space:Diff_surf}.
\end{itemize}
The latter case is quite in the spirit of the fundamental paper~\cite[Chapter~4]{Cerf:BSMF:1961} by J.~P.~Cerf treating the topology of the spaces of ``\myemph{smooth families of smooth maps}'', see also Remark~\ref{rem:Cerf_and_others}.

The obtained results seem to be useful and applicable for concrete computations of the homotopy types of certain spaces of smooth maps, in particular, groups of diffeomorphisms preserving a smooth function or leaves of a foliations.
They could probably be extended to a more general context of manifolds modeled on infinite-dimensional topological vector spaces.

A good survey on the homotopy types of spaces of smooth maps can be found in~\cite{Smith:CM:2010}, see also~\cite{HendersonWest:BAMS:1970, Sakai:JMSJ:1987, Smith:PhD:1993,  AlzaareerSchmeding:EM:2015, Golasinski:MM:2019}, and~\cite{CohenStacey:CM:2004, Stacey:PJM:2005, Stacey:PLMS:2009, Magnot:arxiv:2015, Stacey:GMJ:2017} for the topology of loop spaces.

\subsection*{Main results}
Let $\Mman,\Nman$ be smooth manifolds possibly with corners (see e.g.~\cite[Chapter~1]{MargalefOuterelo:DT:1992}), $0\leq \rx < \sx \leq \infty$, $\sBB\subset\Crm{\rx}{\Mman}{\Nman}$ an open subset with the topology $\Wr{\rx}$, and $\sAA:=\sBB \cap \Crm{\sx}{\Mman}{\Nman}$.
Endow $\sAA$ with the topology $\Wr{\sx}$.

Moreover, for a subset $\Xman\subset\Mman$ and a continuous map $\pfunc:\Xman\to\Nman$ define the following subspaces of $\sBB$ and $\sAA$ with the induced topologies:
\begin{align*}
\BfX &:= \{\gfunc\in\sBB \mid \gfunc = \pfunc \ \text{on} \ \Xman \}, \\
\AfX &:= \BfX \cap \Crm{\sx}{\Mman}{\Nman} \equiv \sAA \cap \BfX.
\end{align*}
In particular, we get two inclusions
\begin{align}
    & \incl: \sAA \subset \sBB, &
    & \incl: \AfX \subset \BfX.
\end{align}

\begin{definition}\label{def:stab_f_near_X_no_L}
Let $\Xman\subset\Mman$ be a closed subset with non-empty interior and $\pfunc\in\Crm{\sx}{\Mman}{\Nman}$, $0\leq\sx\leq\infty$.
We will say that a pair of $\Cr{\sx}$ isotopies
\begin{align}\label{equ:stabilizing_isotopies}
& \LMiso:[0,1]\times\Mman \to \Mman,  &
  \Niso:[0,1]\times\Nman \to \Nman,
\end{align}
\myemph{stabilizes $\pfunc$ near $\Xman$}, whenever 
\begin{enumerate}[label={\rm(\alph*)}]
\item\label{enum:def:stab_f_near_X_no_L:id}
$\LMiso_0 = \id_{\Mman}$, $\Niso_0 = \id_{\Nman}$, 
\item\label{enum:def:stab_f_near_X_no_L:shrinks}
$\LMiso_{\ptt}(\Xman) \subset \Int{\Xman}$ for $0 < \ptt \leq 1$, or equivalently $\Xman \subset \Int{\bigl(\LMiso_{\ptt}^{-1}(\Xman)\bigr)}$;
\item\label{enum:def:stab_f_near_X_no_L:stab}
$\Niso_{\ptt} \circ \pfunc \circ \LMiso_{\ptt} = \pfunc$ on some neighbourhood of $\Xman$ (which may depend on $\ptt$) for $0 < \ptt \leq 1$.
\end{enumerate}
\end{definition}

\begin{remark}\rm
Let~\eqref{equ:stabilizing_isotopies} be a pair of isotopies stabilizing $\pfunc$ near $\Xman$, and $\Niso_{\ptt} \circ \pfunc \circ \LMiso_{\ptt} = \pfunc$ on some neighbourhood $\Uman_{\ptt}$ of $\Xman$ for $\ptt\in(0,1]$.
Let also $\gfunc\in\Crm{\rx}{\Mman}{\Nman}$ be a map such that $\gfunc = \pfunc$ on $\Xman$, and $\gfunc_{\ptt} := \Niso_{\ptt} \circ \gfunc \circ \LMiso_{\ptt}$.
Then it easily follows from conditions~\ref{enum:def:stab_f_near_X_no_L:id}-\ref{enum:def:stab_f_near_X_no_L:stab} that for $0<\ptt\leq1$, $\gfunc_{\ptt} \equiv \pfunc$ on the neighbourhood $\Uman_{\ptt} \cap\Int\bigl(\LMiso_{\ptt}^{-1}(\Xman)\bigr)$ of $\Xman$, see~Lemma~\ref{lm:stab_near_X_prop}.
In particular, $\gfunc_{\ptt}$ is $\Cr{\sx}$ near $\Xman$.
Thus, existence of such isotopies allows to deform (by arbitrary small perturbation, i.e. for arbitrary small $\ptt>0$) any $\Cr{\rx}$ map $\gfunc$ coinciding with $\pfunc$ on $\Xman$ to a map $\gfunc_{\ptt}$ which is $\Cr{\sx}$ near $\Xman$.
This property is a principal technical feature allowing to prove relative variants of our results, see Theorem~\ref{th:B_CsMN__B}\ref{enum:th:B_CsMN__B:rel}.
\end{remark}

\begin{example}\label{exmp:tub_nbh}\rm
Let $\Bman \subset \Mman$ be a submanifold, $\Xman$ a closed tubular neighbourhood of $\Bman$, so there exists a smooth retraction $\rho:\Xman\to\Bman$ having a structure of a locally trivial fibration with a fibre being a closed disk of dimension $\dim\Mman - \dim\Bman$.
Fix any $\Cr{\sx}$ isotopy $\LMiso:[0,1]\times\Mman\to\Mman$ such that $\LMiso_0=\id_{\Mman}$ and $\LMiso_{\ptt}(\Xman) \subset \Int{\Xman}$ for all $t\in(0,1]$.
Let also $\Niso:[0,1]\times\Mman\to\Mman$ be the isotopy consisting of inverses of $\LMiso$, that is $\Niso_{\ptt} = \LMiso_{\ptt}^{-1}$.
The latter implies that $\Niso_{\ptt}\circ\id_{\Mman}\circ\LMiso_{\ptt} =\id_{\Mman}$, i.e. this pair of isotopies stabilizes the identity map $\pfunc=\id_{\Mman}$ near $\Xman$.
\end{example}

\begin{theorem}\label{th:B_CsMN__B}
\begin{enumerate}[leftmargin=*, label={\rm(\arabic*)}]
\item\label{enum:th:B_CsMN__B:abs}
The inclusion $\incl: \sAA \subset \sBB$ is always a weak homotopy equivalence.
If $\Mman$ is compact, then by \cite[Corollary~2]{Milnor:TrAMS:1959}, $\sAA$ and $\sBB$ have homotopy types of CW-complexes, and therefore $\incl$ is in fact a homotopy equivalence.

\item\label{enum:th:B_CsMN__B:rel}
If there exists a pair of isotopies~\eqref{equ:stabilizing_isotopies} stabilizing $\pfunc$ near $\Xman$, then the inclusion $\incl:\AfX \subset \BfX$ is a weak homotopy equivalence as well.
\end{enumerate}
\end{theorem}
Notice that part~\ref{enum:th:B_CsMN__B:abs} is a particular case of~\ref{enum:th:B_CsMN__B:rel} for $\Xman=\varnothing$.
Moreover, Theorem~\ref{th:B_CsMN__B} is a particular case of Theorem~\ref{th:B_CsMN__B_rel_param} below.

For the convenience of the reader we recall a list of useful open sets of $\Cr{\rx}$ maps.

\begin{lemma}\label{lm:open_sets}{\rm(see e.g.~\cite{Hirsch:DiffTop})}
Table~\ref{tbl:open_sets} below contains a list of open subsets of $\Crm{\rx}{\Mman}{\Nman}$ in the topology $\Wr{\rx}$, $\rx\geq 1$.
Hence by Theorem~{\rm\ref{th:B_CsMN__B}\ref{enum:th:B_CsMN__B:abs}} the inclusion $G^{\sx} \subset G^{\rx}$ $(\rx<\sx\leq\infty)$ of the topological spaces from the same line of Table~\ref{tbl:open_sets} is a \myemph{homotopy equivalences}.

\begin{table}[!htpb]
\footnotesize\rm
\begin{tabular}{|c|p{6.3cm}|p{3cm}|} \hline
Set $G^{\rx}$ & Description & Condition \\
    &             & for openness \\  \hline
%%%%%%%%%%%%%%%%%%%%%%%%%%
$\Crm{\rx}{\Mman}{\Nman}$ &
all $\Cr{\rx}$ maps & always \\ \hline
%%%%%%%%%%%%%%%%%%%%%%%%%%
$\Diff^{\rx}(\Mman,\Nman)$ &
diffeomorphisms &
$\Mman$ is closed \\  \hline
%%%%%%%%%%%%%%%%%%%%%%%%%%
$\mathrm{Sub}^{\rx}(\Mman,\Nman)$ &
submersions &
$\Mman$ is compact  \\ \hline
%%%%%%%%%%%%%%%%%%%%%%%%%%
$\mathrm{Imm}_{\Kman}^{\rx}(\Mman,\Nman)$ &
maps $\func:\Mman\to\Nman$ being immersions on a subset $\Kman\subset\Mman$ &
$\Kman$ is compact and is contained in $\Int{\Mman}$ \\ \cline{1-2}
%%%%%%%%%%%%%%%%%%%%%%%%%%
$\mathrm{Emb}_{\Kman}^{\rx}(\Mman,\Nman)$ &
maps $\func:\Mman\to\Nman$ being embeddings on some neighbourhood $\Uman$ (depending on $\func$) of a subset $\Kman\subset\Mman$ &
  \\ \cline{1-2}
%%%%%%%%%%%%%%%%%%%%%%%%%%
$\pitchfork_{\Kman}^{\rx}\!\!(\Mman,\Nman; \Lman)$ &
maps which are transversal on a subset $\Kman \subset \Mman$ to a \myemph{closed} submanifold $\Lman\subset\Jr{r}{\Mman}{\Nman}$ &
                \\ \hline
%%%%%%%%%%%%%%%%%%%%%%%%%%
$\mathrm{Morse}^{\rx}(\Mman,\Nman)$ &
\myemph{Morse} (i.e.\! having only nondegenerate critical points) maps $\Mman\to\Nman$  &
$\rx\geq2$, $\Mman$ is a \myemph{closed} manifold, $\Nman$ is either $\bR$ or the circle $S^1$  \\  \hline
\end{tabular}
\caption{Open sets}\label{tbl:open_sets}
\end{table}
\end{lemma}
\begin{proof}
For $\Crm{\rx}{\Mman}{\Nman}$ the statement is trivial, and the spaces $\mathrm{Imm}_{\Mman}^{\rx}(\Mman,\Nman)$, $\mathrm{Sub}^{\rx}(\Mman,\Nman)$, $\mathrm{Emb}_{\Mman}^{\rx}(\Mman,\Nman)$, $\Diff^{\rx}(\Mman,\Nman)$, and $\mathrm{Morse}^{\rx}(\Mman,\Nman)$ are open in the \myemph{strong} topology $\Sr{\rx}$ on $\Crm{\rx}{\Mman}{\Nman}$, see respectively \cite[Chapter~2, Theorems~1.1, 1.2, 1.4, 1.6]{Hirsch:DiffTop} and \cite[Chapter~6, Theorem~1.2]{Hirsch:DiffTop}.
But that topology coincides with $\Wr{\rx}$ for compact $\Mman$, whence the above sets are also open in $\Wr{\rx}$.

Finally, the sets $\mathrm{Imm}_{\Kman}^{\rx}(\Mman,\Nman)$, $\mathrm{Emb}_{\Kman}^{\rx}(\Mman,\Nman)$, and $\pitchfork_{\Kman}^{\rx}\!\!(\Mman,\Nman; \Lman)$ with compact $\Kman$ are open due to \cite[Chapter~2, \S1, Exercises~13, 14]{Hirsch:DiffTop} and \cite[Chapter 3, Theorem~2.8]{Hirsch:DiffTop}.
\end{proof}

\begin{remark}\rm
In \cite[Chapter~3, Theorem~2.2]{Hirsch:DiffTop} it was introduced a notion of the so-called ``\myemph{rich}'' class of $\Cr{\rx}$ maps $\Mman\to\Nman$ which constitute an open subset of $\Crm{\rx}{\Mman}{\Nman}$.
See also exercises in \cite[Chapters~2 and 3]{Hirsch:DiffTop} for other examples of open subsets of $\Crm{\rx}{\Mman}{\Nman}$.
\end{remark}

\subsection*{Parametrized version of Theorem~\ref{th:B_CsMN__B}}
Let $\Lman$ be another smooth manifold.
Then each $\Cr{\rx}$ map $\Func:\Lman\times\Mman\to\Nman$ induces a map $\func:\Lman \to \Crm{\rx}{\Mman}{\Nman}$ defined by $\func(\ptl)(\ptm) = \Func(\ptl,\ptm)$.
The maps $\Func$ and $\func$ will be called \myemph{adjoint} to each other.
One can check (see Lemma~\ref{lm:exp_law} below) that $\func$ is continuous and the correspondence $\Func\mapsto\func$ is a \myemph{continuous injective} map
\begin{equation}\label{equ:EXP}
	\EXP:\Crm{\rx}{\Lman\times\Mman}{\Nman} \to \Cont{\Lman}{\Crm{\rx}{\Mman}{\Nman}}.
\end{equation}

Conversely, every map $\func:\Lman \to \Crm{\rx}{\Mman}{\Nman}$ induces the adjoint mapping $\Func:\Lman\times\Mman\to\Nman$ by $\Func(\ptl,\ptm) = \func(\ptl)(\ptm)$ which is $\Cr{\rx}$ on sets $\ptl\times\Mman$ for each $\ptl\in\Lman$.
Moreover, continuity of $\func$ means that the ``$\rx$-jet of $\Func$ along $\Mman$'' (i.e. partial derivatives along $\Mman$ of its coordinate functions up to order $\rx$) are continuous on $\Lman\times\Mman$, (see Lemma~\ref{lm:exp_law_finite} and Example~\ref{exmp:C0r_func}).
On the other hand, $\Func$ is not necessarily differentiable in $\Lman$.

Also notice that if $\LMiso:\Lman\times\Mman\to\Lman\times\Mman$ is a $\Cr{\sx}$ map, then the composition $\Func\circ\LMiso$ does not necessarily induces a map $\Lman\to\Crm{\rx}{\Mman}{\Nman}$.
The reason is that $\Func$ is differentiable only along the sets $\{ \ptl\times\Mman \}_{\ptl\in\Lman}$.
Therefore if $\LMiso$ does not preserve the partition into those submanifolds, i.e. $\LMiso(\ptl\times\Mman)$ is not contained in some $\ptl'\times\Mman$, then the restriction $\Func\circ\LMiso|_{\ptl}:\ptl\times\Mman\to\Nman$ might not be differentiable.

Let $\pr{\Lman}:\Lman\times\Mman\to\Lman$ be the natural projection regarded as a trivial fibration.
Then a $\Cr{\rx}$ map $\LMiso=(\Hl,\Hm):\Lman\times\Mman\to\Lman\times\Mman$ will be called a \myemph{morphism} of the fibration $\pr{\Lman}$ whenever
$\pr{\Lman}\circ\LMiso = \Hl\circ\pr{\Lman}$.
This means that $\LMiso(\ptl,\ptm)=(\Hl(\ptl),\Hm(\ptl,\ptm))$, i.e. its first coordinate function $\Hl$ does not depend on $\ptm\in\Mman$ and induces a $\Cr{\rx}$ map $\Hl:\Lman\to\Lman$.
One easily checks that if $\LMiso$ is a morphism of $\pr{\Lman}$, then $\Func\circ\LMiso$ induces a continuous map $\Lman\to\Crm{\rx}{\Mman}{\Nman}$, (see Lemma~\ref{lm:C0r_backward_composition}).
Thus morphisms of the trivial fibration $\pr{\Lman}$ are the \myemph{right} kind of $\Cr{\rx}$ self-maps of $\Lman\times\Mman$ acting on the space of continuous maps $\Lman\to\Crm{\rx}{\Mman}{\Nman}$.

\begin{theorem}\label{th:B_CsMN__B_rel_param}
\begin{enumerate}[wide, label={\rm(\arabic*)}]
\item\label{enum:th:B_CsMN__B_rel_param:abs}
Let $0\leq \rx \leq \sx \leq \infty$, $\sBB \subset \Cont{\Lman}{\Crm{\rx}{\Mman}{\Nman}}$ be an open subset, and
\[ \sAA = \EXP^{-1}(\sBB) \cap \Crm{\sx}{\Lman\times\Mman}{\Nman} \]
be the set of $\Cr{\sx}$ maps $\Lman\times\Mman\to\Nman$ whose adjoint map $\Lman\to\Crm{\rx}{\Mman}{\Nman}$ belongs to $\sBB$.
Endow $\sAA$ with the topology $\Wr{\sx}$.
Then the inclusion $\EXP: \sAA \hookrightarrow \sBB$ is a weak homotopy equivalence.

\item\label{enum:th:B_CsMN__B_rel_param:rel}
Let $\Xman\subset\Lman\times\Mman$ be a closed subset, $\pfunc\in\Crm{\sx}{\Lman\times\Mman}{\Nman}$,
\begin{align*}
	\BfX &= \{\gfunc \in \sBB \mid \gfunc(\ptl)(\ptm) = \pfunc(\ptl,\ptm) \ \text{for} \ (\ptl,\ptm)\in\Xman \}, \\
	\AfX &= \EXP^{-1}(\BfX) \cap \Crm{\sx}{\Lman\times\Mman}{\Nman} =
	\{ \GFunc\in\sAA \mid \GFunc = \pfunc \ \text{on} \ \Xman \}.
\end{align*}
Endow $\BfX$ with the induced compact open topology from $\Cont{\Lman}{\Crm{\rx}{\Mman}{\Nman}}$, and $\AfX$ with the topology $\Wr{\sx}$.
Suppose there exists a pair of $\Cr{\sx}$ isotopies
\begin{align*}
& \LMiso=(\Hl,\Hm):[0,1]\times\Lman\times\Mman \to \Lman\times\Mman,  &
  \Niso:[0,1]\times\Nman \to \Nman,  &
\end{align*}
\myemph{stabilizing $\pfunc$ near $\Xman$} and such that each $\LMiso_{\ptt}$ consists of morphisms of the trivial fibration $\pr{\Lman}$, i.e. its coordinate function $\Hl:[0,1]\times\Lman\times\Mman \to \Lman$ does not depend on $\ptm\in\Mman$.
Then the inclusion $\EXP: \AfX \hookrightarrow \BfX$ is a weak homotopy equivalence.

\item\label{enum:th:B_CsMN__B_rel_param:pairs}
Statements~\ref{enum:th:B_CsMN__B_rel_param:abs} and~\ref{enum:th:B_CsMN__B_rel_param:rel} hold if we replace $\Crm{\rx}{\Mman}{\Nman}$ with the space $\CrmPairs{\rx}{\Mman}{\partial\Mman}{\Nman}{\partial\Nman}$ of $\Cr{\rx}$ maps $\func:\Mman\to\Nman$ such that $\func(\partial\Mman) \subset \partial\Nman$.
\end{enumerate}
\end{theorem}

Again part~\ref{enum:th:B_CsMN__B_rel_param:abs} is a particular case of~\ref{enum:th:B_CsMN__B_rel_param:rel} for $\Xman=\varnothing$.
Moreover, all the Theorem~\ref{th:B_CsMN__B} is a particular case of Theorem~\ref{th:B_CsMN__B_rel_param} for $\Lman$ being a point.
Statement~\ref{enum:th:B_CsMN__B_rel_param:pairs} will be proved in Section~\ref{sect:maps_of_pairs} (Theorem~\ref{th:B_CsMN__B_rel_param_pairs}).

\subsection*{Few applications of Theorem~\ref{th:B_CsMN__B_rel_param}}
Let $\rx \geq 1$, $\Mman$ be a smooth manifold, $\Iman = [0,1]$, $\Jman = [0,0.1] \cup [0.9,1]$, and
\begin{align*}
\LoopDM   &:= \CrmPairs{0}{\Iman}{\partial\Iman}{\Diff^{\rx}(\Mman)}{\id_{\Mman}} \\ 
& \qquad = 
\{ \gamma\in\Cont{\Iman}{\Diff^{\rx}(\Mman)} \mid \gamma(0)=\gamma(1)=\id_{\Mman} \},
\\
\LoopDMJ  &:= \CrmPairs{0}{\Iman}{\Jman}{\Diff^{\rx}(\Mman)}{\id_{\Mman}} \\ 
& \qquad = 
\{ \gamma\in \LoopDMJ   \mid \gamma(\ptt)=\id_{\Mman} \ \text{for all} \ t\in\Jman \}.
\end{align*}
%\begin{align*}
%\LoopDM  & := \CrmPairs{0}{\Iman}{\partial\Iman}{\Diff^{\rx}(\Mman)}{\id_{\Mman}} = 
%\{ \gamma\in\Cont{\Iman}{\Diff^{\rx}(\Mman)} \mid \gamma(0)=\gamma(1)=\id_{\Mman} \}
%\\
%\LoopDMJ & := \CrmPairs{0}{\Iman}{\Jman}{\Diff^{\rx}(\Mman)}{\id_{\Mman}}
%\end{align*}
%\begin{align*}
%\LoopDM  & := \CrmPairs{0}{\Iman}{\partial\Iman}{\Diff^{\rx}(\Mman)}{\{\id_{\Mman}\}}, &
%\LoopDMJ & := \CrmPairs{0}{\Iman}{\Jman}{\Diff^{\rx}(\Mman)}{\{\id_{\Mman}\}}.
%\end{align*}
Thus $\LoopDM$ is the \myemph{loop space} of the group of $\Cr{\rx}$ diffeomorphisms $\Diff^{\rx}(\Mman)$ of $\Mman$ at $\id_{\Mman}$, and $\LoopDMJ$ is the subspace consisting of loops that are constant on $\Jman$, i.e. ``at the beginning and at the end of a path''.
It will be shown in Lemma~\ref{lm:loop_inclusion} below that the inclusion $\LoopDMJ \subset \LoopDM$ is a homotopy equivalence.
Let also
\[
	\DiffIMdIM = \{ \dif\in\Diff^{\rx}(\Iman\times\Mman) \mid \pr{\Iman} \circ \dif = \pr{\Iman} \ \text{and} \ \dif|_{\Jman\times\Mman} = \id_{\Jman\times\Mman} \}
\]
be the group of $\Cr{\rx}$ diffeomorphisms of $\Iman\times\Mman$ which are fixed on $\Jman\times\Mman$ and leave each set $\ptt\times \Mman$ invariant, i.e. preserving the projection $\pr{\Iman}:\Iman\times \Mman \to \Iman$.
Thus each $\Func\in \DiffIMdIM$ is given by $\Func(\ptt,\ptm) = \bigl( \ptt, \eta(\ptt,\ptm) \bigr)$ for some $\Cr{\rx}$ isotopy $\eta:\Iman\times \Mman \to \Mman$ such that $\eta_{\ptt}=\id_{\Mman}$ for $\ptt\in\Jman$.
Hence $\DiffIMdIM$ is homeomorphic with the set of all such isotopies endowed with the topology $\Wr{\rx}$.
It also follows from definitions that $\EXP\bigl(\DiffsIMdIM\bigr) \subset \LoopDMJ$.
\begin{theorem}\label{th:loop_space}
If $\Mman$ is compact, then for $1\leq\rx\leq\sx\leq\infty$ the inclusion
\begin{equation}\label{equ:Diff_in_Loop}
	\EXP: \DiffsIMdIM \ \hookrightarrow \ \LoopDMJ
\end{equation}
is a weak homotopy equivalence.
Hence (by Lemma~\ref{lm:loop_inclusion}) the inclusion \[\EXP: \DiffsIMdIM \ \hookrightarrow \ \LoopDM\] is a weak homotopy equivalence as well.
\end{theorem}
\begin{proof}
Since $\Mman$ is compact, $\Diff^{\rx}(\Mman)$ is open in $\CrmPairs{\rx}{\Mman}{\partial\Mman}{\Mman}{\partial\Mman}$, whence the set $\sBB:= \Cont{\Iman}{\Diff^{\rx}(\Mman)}$ is open in $\Crmx{0}{\Iman}{\CrmPairs{\rx}{\Mman}{\partial\Mman}{\Mman}{\partial\Mman}}$.
Denote $\Xman = \Jman\times\Mman$ and $\pfunc=\id_{\Iman\times\Mman}$.
Then in the notation of Theorem~\ref{th:B_CsMN__B_rel_param},
\begin{align*}
\BfX &:= \LoopDMJ, &
\AfX &:= \DiffsIMdIM.
\end{align*}

Fix a $\Cinfty$ function $\eta:[0,1]\to[0,1]$ such that $\eta[0,0.1] = 0$, $\eta[0.9,1] = 1$, and $\eta$ is strictly monotone on $[0.1,0.9]$.
Define now the following isotopies $\LMiso,\Niso:[0,0.5] \times \Iman\times\Mman \to \Iman\times\Mman$ by
\begin{align*}
\LMiso(\ptt,\ptl,\ptm) &= \bigl( (1-\ptt) \ptl + \ptt \eta(\ptl), \ptm  \bigr), &
\Niso_{\ptt} &= \LMiso_{\ptt}^{-1}.
\end{align*}
Then $\LMiso$ (as well as $\Niso$) consists of automorphisms of the trivial fibration $\pr{\Iman}:\Iman\times\Mman\to\Iman$.

One easily checks that those isotopies stabilize $\pfunc$ near $\Xman$.
Indeed, notice that for $t\in(0,1]$
\begin{align*}
&\LMiso_{\ptt}\bigl([0,0.1]\times\Mman\bigr) \subset [0,0.1) \times\Mman, &
&\LMiso_{\ptt}\bigl([0.9,1]\times\Mman\bigr) \subset (0.9,1]\times\Mman, 
\end{align*}
whence $\LMiso_{\ptt}(\Jman\times\Mman) \subset \Int(\Jman\times\Mman)$. 
Moreover, similarly to Example~\ref{exmp:tub_nbh}, $\Niso_{\ptt} \circ \id_{\Iman\times\Mman} \circ \LMiso_{\ptt} = \id_{\Iman\times\Mman}$.

Hence due to Theorem~\ref{th:B_CsMN__B_rel_param} the inclusion~\eqref{equ:Diff_in_Loop} is a weak homotopy equivalence.
\end{proof}
\begin{corollary}\label{cor:loop_space:Diff_surf}
Let $\Mman$ be a connected compact surface distinct from $S^2$ and $\bR{P}^2$.
Then every connected component of $\DiffIMX{\rx}$, $1\leq\rx\leq\infty$, is weakly contractible (i.e.\! all of its homotopy groups are trivial) with respect to the topology $\Wr{\rx}$.
\end{corollary}
\begin{proof}
Notice that for $k\geq1$ we have the following isomorphisms of homotopy groups:
\begin{align*}
\pi_{k}\bigl( \DiffIMX{\rx}, \id_{\Iman\times\Mman} \bigr) & \cong \pi_k \bigl( \LoopDM, c_{\id_{\Mman}} \bigr) \cong \\ 
&\cong \pi_{k+1}\bigl(\Diff^{\rx}(\Mman), \id_{\Mman}) \cong \pi_{k+1}\bigl(\Diff^{\infty}(\Mman), \id_{\Mman}),
\end{align*}
%\begin{align*}
%	\pi_{k}\DiffIMX{\rx} \cong \pi_k\LoopDM \cong \pi_{k+1}\Diff^{\rx}(\Mman) \cong \pi_{k+1}\Diff^{\infty}(\Mman),
%\end{align*}
%where in the second base points are taken at the corresponding identity maps
where $c_{\id_{\Mman}}:\Iman\to\Diff^{\rx}(\Mman)$ the constant map into the indentity diffeomorphism $\id_{\Mman}$ of $\Mman$,
the first isomorphism is established in Theorem~\ref{th:loop_space}, the second one is well known, e.g.~\cite[Proposition~4.3 \& Corollary~4.8]{Switzer:AT:2002}, and the third one follows from Lemma~\ref{lm:open_sets} since $\Diff^{\rx}(\Mman)$ and $\Diff^{\infty}(\Mman)$ are weakly homotopy equivalent.
But if $\Mman$ is neither $S^2$ nor $\bR{P}^2$, then all path components of $\Diff^{\infty}(\Mman)$ are homotopy equivalent to one of the following aspherical spaces: point, $S^1$ or $S^1\times S^1$, e.g.~\cite{Smale:ProcAMS:1959, EarleEells:BAMS:1967, EarleEells:DG:1970, Gramain:ASENS:1973}.
Hence the above groups vanish.
\end{proof}

Another application of Theorem~\ref{th:loop_space} is based on Example~\ref{exmp:tub_nbh}:
\begin{corollary}\label{cor:tubular_nbh_whe_anyM}
Let $1\leq\rx<\sx\leq\infty$, $\Bman$ be a closed submanifold of a compact manifold $\Mman$, $\Xman$ a tubular neighbourhood of $\Bman$, and $\Diff_{\Xman}^{\rx}(\Mman)$ the group of $\Cr{\rx}$ diffeomorphisms of $\Mman$ fixed on $\Xman$.
For example $\Bman$ can be non-empty collection of boundary components of $\Mman$ and $\Xman$ a collar of $\Bman$.
Then, due to Example~\ref{exmp:tub_nbh}, the inclusion $\Diff_{\Xman}^{\sx}(\Mman) \subset \Diff_{\Xman}^{\rx}(\Mman)$ is a weak homotopy equivalence.
\qed
\end{corollary}

\subsection*{Structure of the paper}
In Section~\ref{sect:weak_homot_equiv} we give a characterization of weak homotopy equivalences which are injective maps (Lemma~\ref{lm:char_whe}).
Section~\ref{sect:preliminaries} contains a list of basic properties of compact open topologies (Lemma~\ref{lm:cotop}), and the proof of exponential law for non-Hausdorff spaces (Lemma~\ref{lm:exp_law}).
In Section~\ref{sect:whitney_topologies} we discuss Whitney topologies on the spaces of smooth maps between manifolds.

In Section~\ref{sect:maps_from_prods} we consider the set $\Crm{0,\rx}{\Lman\times\Mman}{\Nman}$ of adjoint maps to continuous maps from $\Cont{\Lman}{\Crm{\rx}{\Mman}{\Nman}}$, and similarly to~\cite{Amann:ODE:1990, Glockner:JFA:2002, AlzaareerSchmeding:EM:2015}, endow the space $\Crm{0,\rx}{\Lman\times\Mman}{\Nman}$ with a topology (denoted $\Wr{0,\rx}$) induced from $\Cont{\Lman}{\Crm{\rx}{\Mman}{\Nman}}$.
We prove continuity of the compositions for maps from $\Crm{0,\rx}{\Lman\times\Mman}{\Nman}$ (Lemmas~\ref{lm:C0r_forward_composition} and~\ref{lm:C0r_backward_composition}), an exponential law (Lemma~\ref{lm:exp_law_C0r}), and also formulate analogues of smooth approximation results.

In Section~\ref{sect:diff_approx} we establish the main approximation Lemma~\ref{lm:approx_C0r}, and in Section~\ref{sect:proof:th:B_CsMN__B_rel_param} prove statements~\ref{enum:th:B_CsMN__B_rel_param:abs} and~\ref{enum:th:B_CsMN__B_rel_param:rel} of Theorem~\ref{th:B_CsMN__B_rel_param}.
Finally, in Section~\ref{sect:maps_of_pairs} we show how to extend the obtained results to opens subsets of the spaces of preserving boundaries maps $\CrmPairs{\rx}{\Mman}{\partial\Mman}{\Nman}{\partial\Nman}$, e.g. to groups of diffeomorphisms of manifolds with boundary (Theorem~\ref{th:B_CsMN__B_rel_param_pairs}).
In particular, this will prove statement~\ref{enum:th:B_CsMN__B_rel_param:pairs} of Theorem~\ref{th:B_CsMN__B_rel_param}.

\section{Weak homotopy equivalences}\label{sect:weak_homot_equiv}
In this section we present a simple characterization of injective continuous maps being weak homotopy equivalences.
Such a question is discussed in~\cite[\S6]{Glockner:0812.4713:2008} for a special type of spaces, see Remark~\ref{rem:comparison_with_Glockner}.

Let $\Disk{k+1}$ be the unit $k$-dimensional disk in $\bR^{k+1}$, $S^k = \partial \Disk{k+1}$ the $k$-dimensional sphere, and $\bpt=(1,0,\ldots,0)\in S^k$.
Let also $(\sBB,\tau)$ be a topological space, $\sAA$ a subset of $\sBB$ endowed with the induced topology $\tau_{\sAA}$, $\dif\in\sAA$, and $\incl:\sAA\subset\sBB$ the inclusion map.
Then it is easy to see that the following conditions are equivalent:
\begin{enumerate}[leftmargin=*, label={\rm(\arabic*)}]
\item
for each $k\geq0$ the induced homomorphism $\incl_k:\pi_k(\sAA,\dif) \to \pi_k(\sBB,\dif)$ of the corresponding homotopy groups is an isomorphism (bijection of sets for $k=0$);

\item
for each $k\geq1$ the corresponding relative homotopy sets are trivial, i.e. $\pi_k(\sBB,\sAA,\dif) = 0$;

\item
every continuous map of triples $\beta:(\Disk{k+1},S^k,\bpt) \to (\sBB,\sAA,\dif)$ is homotopic as a map of \myemph{triples} to a map into $\sAA$;

\item
every continuous map of triples $\beta:(\Disk{k+1},S^k,\bpt) \to (\sBB,\sAA,\dif)$ is homotopic as a map of \myemph{pairs} $(\Disk{k+1},S^k) \to (\sBB,\sAA)$ to a map into $\sAA$.
\end{enumerate}
If either of these conditions holds, then $\incl:\sAA\subset\sBB$ is called a \myemph{weak homotopy equivalences}.

\medskip

Suppose now that $\sAA$ is endowed with some other topology $\sigma$.
Then $\incl:(\sAA,\sigma)\subset(\sBB,\tau)$ is still continuous if and only if $\sigma$ is stronger than $\tau_{\sAA}$.

Thus suppose that $\incl$ is continuous.
Then a continuous map of pairs
\[\beta:(\Disk{k+1}, S^{k}) \to (\sBB,\sAA)\]
will be called \myemph{$\sigma$-admissible}%
\footnote{This property is very close to the notion of a \myemph{compactly retractable} topological space defined in~\cite[Definition~3.2]{Glockner:0812.4713:2008}.}%
, if the restriction $\beta|_{S^{k}}:S^{k}\to\sAA$ is still continuous into the topology $\sigma$ of $\sAA$.
More formally, there should exist a continuous map $\alpha:S^k \to \sAA$ such that $\beta|_{S^{k}} = \incl\circ\alpha$, and so we get the following commutative diagram of continuous maps:
\[
\xymatrix{
(\sAA,\sigma) \ar[r]^-{\incl} & (\sBB,\tau) \\
\ S^k \ \ar@{^(->}[r] \ar[u]^-{\alpha} & \ \Disk{k+1} \ \ar[u]^-{\beta}
}
\]
A homotopy of pairs $H:[0,1]\times \Disk{k+1} \to \sBB$ will be called \myemph{$\sigma$-admissible} whenever $H([0,1]\times S^k) \subset \sAA$, and the restriction $H: [0,1]\times S^k \to \sAA$ is continuous into the topology $\sigma$.
In particular, each map $H_t$ is $\sigma$-admissible.

The following lemma is a simple exercise in homotopy groups.
\begin{lemma}\label{lm:char_whe}
Let $(\sAA,\sigma)$ and $(\sBB,\tau)$ be path connected topological spaces and $\incl:\sAA \to \sBB$ be an \myemph{injective continuous} map,
so we can identify $\sAA$ with a subset of $\sBB$, but endow it with a stronger than $\tau_{\sAA}$ topology $\sigma$.
Let also $\dif\in\sAA$.
Then the following conditions are equivalent:
\begin{enumerate}[leftmargin=*, label={\rm(\arabic*)}, itemsep=0.4ex]
\item\label{enum:lm:char_whe:whe}
$\incl:\sAA \to \sBB$ is a weak homotopy equivalence;

\item\label{enum:lm:char_whe:adm_maps_A}
every $\sigma$-admissible map $\beta:(\Disk{k+1}, S^{k}, \bpt) \to (\sBB,\sAA,\dif)$ is homotopic relatively to $S^{k}$ to a map into $\sAA$;

\item\label{enum:lm:char_whe:adm_maps}
every $\sigma$-admissible map $\beta:(\Disk{k+1}, S^{k}, \bpt) \to (\sBB,\sAA,\dif)$ is homotopic by a $\sigma$-admissible homotopy of \myemph{triples} $(\Disk{k+1}, S^{k}, \bpt) \to (\sBB,\sAA,\dif)$ to a map into $\sAA$;

\item\label{enum:lm:char_whe:adm_maps_pairs}
every $\sigma$-admissible map $\beta:(\Disk{k+1}, S^{k}, \bpt) \to (\sBB,\sAA,\dif)$ is homotopic by a $\sigma$-admissible homotopy of \myemph{pairs} $(\Disk{k+1}, S^{k}) \to (\sBB,\sAA)$ to a map into $\sAA$.
\end{enumerate}
\end{lemma}
\begin{proof}
\ref{enum:lm:char_whe:whe}$\Rightarrow$\ref{enum:lm:char_whe:adm_maps}
Let $\beta:(\Disk{k+1}, S^{k}, \bpt) \to (\sBB,\sAA,\dif)$ be a $\sigma$-admissible map, and $\alpha:S^{k} \to (\sAA,\sigma)$ the map such that $\beta|_{S^k} = \incl\circ\alpha$.
In particular, we see that $\incl_{k}\bigl([\alpha]\bigr)=0$ in $\pi_k(\sBB,\dif)$.
Since $\incl_{k}$ is an isomorphism, $[\alpha]=0$ in $\pi_k(\sAA,\dif)$, i.e. there exists a homotopy $\AFunc:[0,1]\times S^k \to \sAA$ such that $\AFunc_0=\alpha$, $\AFunc_{\ptt}(\bpt)=\dif$ for all $\ptt\in[0,1]$, and $\AFunc_1(S^k)=\{\dif\}$.
Since the pair $(\Disk{k+1},S^k)$ satisfies homotopy extension axiom, the homotopy $\incl\circ \AFunc:[0,1]\times S^k \to \sBB$ extends to some homotopy $\BFunc:[0,1]\times \Disk{k+1} \to \sBB$ such that $\BFunc_0=\beta$.

The assumption $\BFunc|_{[0,1]\times S^{k}} = \incl\circ \AFunc$ means that the homotopy $\BFunc$ consists of $\sigma$-admissible maps.
Therefore we can replace $\beta$ with $\BFunc_1$ and assume that $\beta(S^k) = \BFunc_1(S^k)= \{\dif\}$.

Thus $\beta$ is a map of pairs $\beta:(\Disk{k+1},S^k) \to (\sBB, \{\dif\})$, and therefore it represents some element $[\beta]$ of $\pi_{k+1}(\sBB,\dif)$.
Since $\incl_{k+1}$ is an isomorphism as well, $\beta$ is homotopic relatively $S^k$ (and therefore by a $\sigma$-admissible homotopy) to a map into $\sAA$.

\ref{enum:lm:char_whe:adm_maps_A}$\Rightarrow$\ref{enum:lm:char_whe:adm_maps}
If $\beta:(\Disk{k+1}, S^{k}, \bpt) \to (\sBB,\sAA,\dif)$ is a $\sigma$-admissible map, then any homotopy of $\beta$ relatively to $S^{k}$ is $\sigma$-admissible.

\ref{enum:lm:char_whe:adm_maps}$\Rightarrow$\ref{enum:lm:char_whe:adm_maps_pairs}
This implication follows from the observation that every ($\sigma$-admissible) homotopy of triples $(\Disk{k+1}, S^{k}, \bpt) \to (\sBB,\sAA,\dif)$ is also a ($\sigma$-admissible) homotopy of pairs $(\Disk{k+1}, S^{k}) \to (\sBB,\sAA)$.

\ref{enum:lm:char_whe:adm_maps_pairs}$\Rightarrow$\ref{enum:lm:char_whe:adm_maps_A}
Let $\BFunc:[0,1]\times \Disk{k+1} \to \sBB$ be a $\sigma$-admissible homotopy such that $\beta = \BFunc_0$ and $\BFunc_1(\Disk{k+1}) \subset \sAA$.
In particular, $\BFunc\bigl([0,1]\times S^{k}\bigr) \subset \sAA$.
Fix any continuous map $\eta: [0,1]\times \Disk{k+1} \to [0,1]\times \Disk{k+1}$ having the following properties:
\begin{enumerate}[label={\rm(\arabic*)}, itemsep=0.5ex]
\item
$\eta$ is fixed on $0\times \Disk{k+1}$;

\item
$\eta([0,1]\times\pts)=(0,\pts)$ for all $\pts\in S^k$;

\item
$\eta(1\times \Disk{k+1}) \subset \bigl(1\times \Disk{k+1}\bigr) \cup \bigl([0,1]\times S^{k}\bigr)$;

\end{enumerate}
Then $\BFunc\circ\eta:[0,1]\times \Disk{k+1} \to \sBB$ is a $\sigma$-admissible homotopy relatively $S^k$ between $\BFunc_0$ and a map into $\sAA$.

\ref{enum:lm:char_whe:adm_maps_A}$\Rightarrow$\ref{enum:lm:char_whe:whe}
\myemph{Surjectivity of $\incl_{k}$.}
Let $\beta:(\Disk{k},S^{k-1})\to  (\sBB, \{\dif\})$ be a continuous map.
We have to show that $\beta$ is homotopic as a map of pairs to a map $\incl\circ\alpha$ for some map $\alpha:(\Disk{k},S^{k-1})\to  (\sAA, \{\dif\})$.
Evidently, $\beta$ can also be regarded as a map of triples $\beta: (\Disk{k},S^{k-1},\bpt)\to(\sBB,\sAA,\{\dif\})$.
Moreover, as $\beta|_{S^k}: S^{k-1} \to \sAA$ is constant, it is also continuous into the topology $\sigma$, whence $\beta$ is a $\sigma$-admissible map.
Hence, by~\ref{enum:lm:char_whe:adm_maps_A}, $\beta$ is homotopic relatively to $S^k$ (i.e.\! sending $S^k$ to a point $\dif$) to a map $\incl\circ\alpha$ for some continuous map $\alpha:(\Disk{k},S^{k-1})\to (\sBB,\{\dif\})$.
Then $\incl_k[\alpha] = \beta$.

\myemph{Injectivity of $\incl_{k}$.}
Let $\alpha:(S^{k},\bpt)\to (\sAA, \{\dif\})$ be a continuous map such that $\incl\circ\alpha:(S^{k},\bpt)\to (\sBB, \{\dif\})$ extends to a continuous map $\beta:\Disk{k+1}\to\sBB$.
We have to show that then $\alpha$ itself extends to a continuous map $\Disk{k+1}\to\sAA$.

Notice that the assumption on $\alpha$ and $\beta$ means that $\beta$ is a $\sigma$-admissible map.
Therefore by~\ref{enum:lm:char_whe:adm_maps_A}, $\beta$ is homotopic relatively to $S^{k}$ to a map $\incl\circ\beta'$ for some continuous map $\beta':\Disk{k} \to \sAA$.
But $\beta'|_{S^k}=\beta|_{S^k}=\alpha$, which means that $\beta'$ is a desired extension of $\alpha$.
\end{proof}

\section{Compact open topologies}\label{sect:preliminaries}

\subsection*{Locally compact spaces}
We will say that a (not necessary Hausdorff) topological space $\Mman$ is \myemph{locally compact} whenever
for each $x\in\Mman$ and an open neighbourhood $\Uman$ of $x$ there exists a compact subset $\Bman\subset\Mman$ such that $x\in\Int{\Bman} \subset\Bman \subset \Uman$.

\begin{lemma}\label{lm:loc_comp_cover}
Let $\Kman$ be a compact subset of a locally compact topological space $\Mman$, and $\{\Uman_i\}_{i=1}^{n}$ be a finite open cover of $\Kman$.
Then there exists finitely many compact sets $\Kman_1,\ldots,\Kman_n$ such that $\Kman_i \subset \Uman_i$ for each $i=1,\ldots,n$ and $\Kman\subset\mathop{\cup}\limits_{i=1}^{n}\Kman_i$.
\end{lemma}
\begin{proof}
For each $x\in\Kman$ fix an element $\Uman_{i_x}$ containing $x$.
By local compactness of $\Mman$ there exists a compact subset $\Bman_x$ such that $x\in\Int{\Bman_x} \subset\Bman_x \subset \Uman_{i_x}$.
Then $\{\Int{\Bman_x}\}_{x\in\Kman}$ is an open cover of $\Kman$, whence there are finitely many points $x_1,\ldots,x_k$ such that $\Kman\subset \mathop{\cup}\limits_{j=1}^{k}\Int{\Bman_{x_j}}$.
Let $\Kman_i = \mathop{\cup}\limits_{i_{x_j} = i} \Bman_{x_j}$ be the union of those $\Bman_{x_j}$ which are contained in $\Uman_{i}$.
Then $\Kman_i$ is compact, $\Kman_i \subset \Uman_i$, and $\Kman\subset\mathop{\cup}\limits_{i=1}^{n}\Kman_i$.
\end{proof}

\subsection*{Compact open topologies}
Recall, e.g.~\cite{Fox:BAMS:1945}, that for any topological spaces $\Mman$ and $\Nman$ the \myemph{compact open} topology on the space $\Cont{\Mman}{\Nman}$ of all continuous maps $\Mman\to\Nman$ is the topology generated as a subbase by sets of the form:
\[
\WNBH{0}{\Xman}{\Vman} = \{ \dif\in\Cont{\Mman}{\Nman} \mid \dif(\Xman)\subset\Vman \},
\]
where $\Xman\subset\Mman$ runs over all compact subsets of $\Mman$ and $\Vman\subset\Nman$ runs over all open subsets of $\Nman$.

In what follows we will always assume that the space $\Cont{\Mman}{\Nman}$ is endowed with the compact open topology.

To make the paper as self contained as possible we will collect in the following lemma several simple properties of compact open topologies.
Most of them are rather trivial and can be found in many sources.
Emphasize, that locally compact spaces in this lemma are not required to be Hausdorff.

\begin{lemma}\label{lm:cotop}
For three topological spaces $\Lman, \Mman,\Nman$ the following statements hold true.
\begin{enumerate}[leftmargin=*, label={\rm(\arabic*)}]

\item\label{enum:lm:cotop:subbase_prop}
\myemph{Properties of subbase.}
\begin{enumerate}[leftmargin=*, label={\rm(\alph*)}, ref={\rm(1\alph*)}]
\item\label{enum:lm:cotop::subbase_prop:a}
$\WNBH{0}{\Kman}{\Vman_1} \cap \WNBH{0}{\Kman}{\Vman_2} = \WNBH{0}{\Kman}{\Vman_1\cap\Vman_2}$,
$\mathop{\cup}\limits_{i\in\Lambda}\WNBH{0}{\Kman}{\Vman_i} = \WNBH{0}{\Kman}{\mathop{\cup}\limits_{i\in\Lambda}\Vman_i}$
for any compact $\Kman\subset\Mman$, open $\Vman_1, \Vman_2 \subset \Nman$ and a family $\{\Vman_i\}_{i\in\Lambda}$ of open subsets of $\Nman$.

\item\label{enum:lm:cotop::subbase_prop:open_subset}
For each open $\Vman\subset\Nman$ the space $\Cont{\Mman}{\Vman}$ is an open subset of $\Cont{\Mman}{\Nman}$.

\item\label{enum:lm:cotop::subbase_prop:closure_KU}
Let $\Uman\subset\Vman$ be open subsets of $\Mman$ such that $\overline{\Uman}\subset\Vman$.
Then for any compact $\Kman\subset\Lman$ we have that
$\overline{[\Kman, \Uman]} \subset [\Kman,\Vman]$, {\rm(\cite[\S46, Excercise~6]{Munkres:Top:2000})}.

\end{enumerate}

\item\label{enum:lm:cotop::separate}
\myemph{Separation axioms.}
If $\Mman$ is a $T_i$-space for some $i=0,1,2,3$, then so is $\Cont{\Lman}{\Mman}$.

%%%%%%%%%%%%%%%%%%%%%
\item\label{enum:lm:cotop::composition}
\myemph{Composition of maps.}
For any continuous $\Func:\Lman\to\Mman$ and $\GFunc:\Mman\to\Nman$ the maps
\begin{align}
\label{equ:indmap}
&\indmap{\GFunc}:\Cont{\Lman}{\Mman} \to \Cont{\Lman}{\Nman}, & \indmap{\GFunc}(\func)  &= \GFunc\circ\func, \\
\label{equ:indmapx}
&\indmapx{\Func}:\Cont{\Mman}{\Nman} \to \Cont{\Lman}{\Nman}, & \indmapx{\Func}(\gfunc) &= \gfunc\circ\Func,
\end{align}
are continuous.
Moreover, if $\Mman$ is \myemph{locally compact}, then the composition map
\[\mult:\Cont{\Lman}{\Mman} \times \Cont{\Mman}{\Nman} \to \Cont{\Lman}{\Nman}, \qquad \mult(\func,\gfunc)=\gfunc\circ\func, \]
is continuous as well.

%%%%%%%%%%%%%%%%%%%%%%%
\item\label{enum:lm:cotop::evaluation}
For a point $x\in\Mman$ let $c_{x}:\Lman \to \Mman$ be the constant map into the point $x$.
Then the map $c:\Mman \to \Cont{\Lman}{\Mman}$, $c(x)=c_x$, is continuous.

\item\label{enum:lm:cotop:map_into_product}
Suppose $\Nman$ is homeomorphic to a product $\prod_{i\in\Lambda}\Nman_i$ of arbitrary family of topological spaces.
Let also $p_j:\prod_{i\in\Lambda}\Nman_i\to \Nman_j$, $j\in\Lambda$, be the natural projection onto $j$-th multiple.
Then the natural map
\[
\eta = \prod_{i\in\Lambda} \indmap{(p_i)}: \Cont{\Mman}{\Nman} \to \prod_{i\in\Lambda} \Cont{\Mman}{\Nman_i},
\qquad
\eta(\func) = \{ p_i\circ\func \}_{i\in\Lambda},
\]
is a continuous bijection.
If either $\Mman$ is \myemph{locally compact} or \myemph{each $\Nman_i$ is a $T_3$ space}, then $\eta$ is a homeomorphism.
\end{enumerate}
\end{lemma}
%%%%%%%%%%%%%%%%%%%%%%%%%%%%%%%%%%%%%%%%%%%%%%%%%%%%%%%%%%%%%%%%%%%
\begin{proof}
\ref{enum:lm:cotop:subbase_prop}
Properties~\ref{enum:lm:cotop::subbase_prop:a} and~\ref{enum:lm:cotop::subbase_prop:open_subset} are trivial.

\ref{enum:lm:cotop::subbase_prop:closure_KU}
Suppose $\Kman\subset\Lman$ is compact, $\Uman,\Vman\subset\Mman$ are open, and $\overline{\Uman}\subset\Vman$.
For the proof that $\overline{[\Kman, \Uman]} \subset [\Kman,\Vman]$ it suffices to show that if $\func\in\Cont{\Lman}{\Mman}\setminus [\Kman,\Vman]$, then $\func\not\in\overline{[\Kman, \Uman]}$ as well, i.e. there exists an neighbourhood $\mathcal{U}$ of $\func$ such that $\mathcal{U} \cap [\Kman, \Uman] = \varnothing$.

The assumption $\func\not\in[\Kman,\Vman]$, i.e. $\func(\Kman)\not\subset\Vman$, implies that 
\[ \Bman := \func(\Kman) \cap (\Mman\setminus\Vman) \not=\varnothing.\]
Therefore $\Aman:=\Kman\cap\func^{-1}(\Bman) = \Kman\cap \func^{-1}(\Mman\setminus\Vman)$ is closed in $\Kman$, and thus it is compact.
Moreover, $\func(\Aman) \subset \Mman\setminus\Vman \subset \Mman\setminus\overline{\Uman}$.
In other words, $\mathcal{U}:=[\Aman,\Mman\setminus\overline{\Uman}]$ is an open neighbourhood of $\func$.
Now for any $\gfunc\in\mathcal{U}$ we have that 
\[\gfunc(\Kman) \supset \gfunc(\Aman) \subset \Mman\setminus\overline{\Uman},\]
whence $\gfunc(\Kman) \not\subset\Uman$, i.e. $\gfunc\not\in[\Kman, \Uman]$.
Thus $\mathcal{U} \cap [\Kman, \Uman] = \varnothing$.

\ref{enum:lm:cotop::separate}
Let $\func\not=\gfunc \in\Cont{\Lman}{\Mman}$, so there exists a point $x\in\Lman$ such that $\func(x)\not=\gfunc(x)$.
If $\Mman$ is $T_0$, there is an open neighbourhoods $\Uman_{\func(x)}$ of $\func(x)$ such that $\gfunc(x)\not\in\Uman_{\func(x)}$.
Then $[\{x\}, \Uman_{\func(x)}]$ is an open neighbourhood of $\func$ which does not contain $\gfunc$.
Hence $\Cont{\Lman}{\Mman}$ is a $T_0$ space as well.
The proof of $T_1$- and $T_2$-cases of $\Mman$ is similar and we leave it for the reader.

Suppose $\Mman$ is $T_3$.
Let $\func\in\Cont{\Lman}{\Mman}$ and $\VNbh$ be its open neighbourhood.
We should find another open neighbourhood $\UNbh$ of $\func$ such that $\overline{\UNbh} \subset \VNbh$.
Decreasing $\VNbh$ if necessary, one can assume that $\VNbh = \cap_{i=1}^{n}[\Kman_i,\Vman_i]$ for some finite $n$, compacts $\Kman_i\subset\Lman$, and open $\Vman_i\subset\Mman$, $i=1,\ldots,n$.
Thus $\func(\Kman_i) \subset \Vman_i$ for $i=1,\ldots,n$.
Since $\Mman$ is $T_3$ space, and $\func(\Kman_i)$ is compact, there exists an open $\Uman_i\subset\Mman$ such that $\func(\Kman_i) \subset \Uman_i \subset \overline{\Uman_i} \subset \Vman_i$.
Put $\UNbh = \cap_{i=1}^{n}[\Kman_i,\Uman_i]$.
Then $\func\in\UNbh$.
Moreover by~\ref{enum:lm:cotop::subbase_prop:closure_KU},
\[
 \overline{\UNbh}
  \ =       \ \overline{\mathop{\cap}\limits_{i=1}^{n}[\Kman_i,\Uman_i]}
  \ \subset \ \mathop{\cap}\limits_{i=1}^{n}\overline{[\Kman_i,\Uman_i]}
  \ \subset \ \mathop{\cap}\limits_{i=1}^{n}[\Kman_i,\Vman_i]
  \ =       \ \VNbh.
\]

%%%%% composition
\ref{enum:lm:cotop::composition}
Let $[\Aman,\Wman]$ be a subbase open set of $\Cont{\Lman}{\Nman}$, where $\Aman \subset\Lman$ is compact, and $\Wman\subset\Nman$ is open.
Then it is easy to see that
\begin{align*}
(\indmap{\GFunc})^{-1}[\Aman,\Wman] &= [\Aman,\GFunc^{-1}(\Wman)], &
(\indmapx{\Func})^{-1}[\Aman,\Wman] &= [\Func(\Aman),\Wman],
\end{align*}
so those inverse images are open in $\Cont{\Lman}{\Mman}$ and $\Cont{\Mman}{\Nman}$ respectively, and therefore $\indmap{\GFunc}$ and $\indmapx{\Func}$ are continuous.
For the proof of continuity of $\mult$, see e.g.\cite{MaksymenkoPolulyakh:grp_act:2020}.

%%%%%%%%%%%%%%%
\ref{enum:lm:cotop::evaluation}
Let $[\Aman,\Vman]$ be the subbase subset of $\Cont{\Lman}{\Mman}$.
Then $c^{-1}\bigl( [\Aman,\Vman] \bigr) = \Vman$, whence $c$ is continuous.

%%%%%%%%%%%%%%%
\ref{enum:lm:cotop:map_into_product}
It is evident that $\eta$ is a bijection.
Moreover, it is continuous as a product of continuous maps $\indmap{(p_i)}$ of type~\eqref{equ:indmap}.

Suppose that either
\begin{enumerate}[label={\rm(5\alph*)}]
\item\label{enum:lm:cotop:map_into_product:assump:a} $\Mman$ is locally compact, or
\item\label{enum:lm:cotop:map_into_product:assump:b} each $\Nman_i$ is a $T_3$ space.
\end{enumerate}
Let $\func\in\Cont{\Mman}{\Nman}$ and $[\Kman,\Uman]$ be a subbase neighbourhood of $\func$, that is $\func(\Kman) \subset \Uman$.
We need to show that $\eta([\Kman,\Uman])$ is open in $\prod_{i\in\Lambda}\Cont{\Mman}{\Nman_i}$.

\begin{sublemma}\label{lm:map_into_product:cond_K}
In both cases~\ref{enum:lm:cotop:map_into_product:assump:a} and~\ref{enum:lm:cotop:map_into_product:assump:b} there exist
\begin{enumerate}[label={\rm(\roman*)}]
\item\label{enum:lm:cotop:map_into_product:cond:1} finite subset of indices $\Sigma\subset\Lambda$ and $s>0$,
\item\label{enum:lm:cotop:map_into_product:cond:2} open sets $\Vman_{1,j},\ldots,\Vman_{s,j} \subset \Nman_j$, $j\in\Sigma$,
\item\label{enum:lm:cotop:map_into_product:cond:3} compact sets $\Kman_j \subset \Mman$, $j\in\Sigma$,
\end{enumerate}
such that
\begin{align*}
\Kman &\subset \mathop{\cup}\limits_{a=1}^{\sx} \Kman_a,  &
\func(\Kman_a) & \ \subset  \
\prod_{j\in\Sigma} \Vman_{a,j} \!\times\! \prod_{i\in\Lambda \setminus \Sigma} \Nman_{i} \ \subset \ \Uman, \ (a=1,\ldots,s).
\end{align*}
\end{sublemma}
\begin{proof}
\ref{enum:lm:cotop:map_into_product:assump:a}
Since $\func(\Kman)$ is compact, there exist finite subset of indices $\Sigma\subset\Lambda$, $s>0$, and open sets $\Vman_{1,j},\ldots,\Vman_{s,j} \subset \Nman_j$, $j\in\Sigma$, such that if we denote $\Vman_a = \prod\limits_{j\in\Sigma} \Vman_{a,j} \,\times\! \prod\limits_{i\in\Lambda \setminus \Sigma} \Nman_{i}$, $(a=1,\ldots,s)$, then $\func(\Kman) \subset \mathop{\cup}\limits_{a=1}^{s} \Vman_a  \subset  \Uman$.
Hence $\{ \func^{-1}(\Vman_a) \}_{a=1}^{s}$ is an open cover of $\Kman$, and due to Lemma~\ref{lm:loc_comp_cover} there are compact subsets $\Kman_a\subset \func^{-1}(\Vman_a)$ such that $\Kman \subset \mathop{\cup}\limits_{a=1}^{\sx} \Kman_a$.

\ref{enum:lm:cotop:map_into_product:assump:b}
Again as $\func(\Kman)$ is compact, and $\Nman$ is $T_3$, there exist finite subset of indices $\Sigma\subset\Lambda$, $s>0$, and open sets $\Vman_{1,j},\ldots,\Vman_{s,j}, \Wman_{1,j},\ldots,\Wman_{s,j} \subset \Nman_j$, $j\in\Sigma$, such that
\begin{itemize}[wide]
\item
$\overline{\Vman_{a,j}} \subset \Wman_{a,j}$ for all $a=1,\ldots,s$ and $j\in\Sigma$;

\item
if we denote $\Vman_a = \prod\limits_{j\in\Sigma} \Vman_{a,j} \,\times\! \prod\limits_{i\in\Lambda \setminus \Sigma} \Nman_{i}$ \ and \ $\Wman_a = \prod\limits_{j\in\Sigma} \Wman_{a,j} \,\times\! \prod\limits_{i\in\Lambda \setminus \Sigma} \Nman_{i}$, $(a=1,\ldots,s)$, then
\[ \func(\Kman)
   \ \subset \ \mathop{\cup}\limits_{a=1}^{s} \Vman_a
   \ \subset \ \mathop{\cup}\limits_{a=1}^{s} \overline{\Vman_a}
   \ \subset \ \mathop{\cup}\limits_{a=1}^{s} \Wman_a
   \ \subset \ \Uman.
\]
\end{itemize}
Hence each $\Kman_a = \func^{-1}(\overline{\Vman_a})\cap\Kman$, $(a=1,\ldots,s)$, is a closed subset of the compact $\Kman$, and therefore it is compact as well.
\end{proof}

Assuming that conditions~\ref{enum:lm:cotop:map_into_product:cond:1}-\ref{enum:lm:cotop:map_into_product:cond:3} of Lemma~\ref{lm:map_into_product:cond_K} hold define the following open subset
\[
\UU = \bigcap\limits_{a=1}^{\sx} \Bigl(
\prod_{j\in\Sigma} [\Kman_a, \Vman_{a,j}]
\times
\prod_{i\in\Lambda \setminus \Sigma} \Cont{\Mman}{\Nman_{i}} \Bigr)
\]
of $\prod_{i\in\Lambda} \Cont{\Mman}{\Nman_i}$.
If $\{ \gfunc_i\}_{i\in\Lambda} \in \UU$, and $\gfunc = \eta^{-1}(\{\gfunc_i\}_{i\in\Lambda})$, then
\[
\gfunc(\Kman) \ \subset \
\gfunc\bigl( \mathop{\cup}\limits_{a=1}^{\sx} \Kman_a \bigr) \ \subset \
\bigcup_{a=1}^{\sx} \Bigl(
\prod_{j\in\Sigma} \Vman_{a,j} \!\times\! \prod_{i\in\Lambda \setminus \Sigma} \Nman_{i}
\Bigr) \ \subset \ \Uman,
\]
that is $\eta^{-1}(\gfunc) \in [\Kman,\Uman]$, and thus $\UU \subset \eta([\Kman,\Uman])$.
In other words, $\eta([\Kman,\Uman])$ is an open neighbourhood of $\eta(\func)$.
This implies that $\eta$ is also an open map, and therefore a homeomorphism.
\end{proof}

\subsection*{Exponential law}\label{sect:exp_law}
We will present here a variant of the \myemph{exponential law} for non-Hausdorff spaces which is usually formulated for Hausdorff spaces only, see Lemma~\ref{lm:exp_law}.
The proof of that lemma is rather standard, however to make the paper self-contained and to be sure that Hausdorff property is avoided we present a detailed and elementary proof.
Such a splitting of exponential law into separate statements can be found e.g. in~\cite[pages 30-37]{Postnikov:Homotopy:1:1984}, or \cite[Theorem 46.11]{Munkres:Top:2000}.
This is motivated by future applications to computations of homotopy types of groups of diffeomorphisms and homeomorphisms of foliations and their relations to the groups of homeomorphisms of the corresponding spaces of leaves.
The latter spaces are often non Hausdorff manifolds, see e.g.~\cite{GodbillonReeb:EM:1966, MaksymenkoPolulyakh:PGC:2015, MaksymenkoPolulyakh:MFAT:2016}.

For two sets $\Mman,\Nman$ we will denote by $\Map(\Mman,\Nman)$ the set of all maps $\Mman\to\Nman$.
If $\Lman$ is another set, then there is a \myemph{bijection}
\begin{equation}\label{equ:exp_law_for_sets}
	\EXP_{\Lman,\Mman,\Nman}:\Map(\Lman\times\Mman,\Nman) \to \Map(\Lman, \Map(\Mman,\Nman))
\end{equation}
given by the following rule: if $\Func\in\Map(\Lman\times\Mman,\Nman)$, then the mapping $\EXP_{\Lman,\Mman,\Nman}(\Func): \Lman \to \Map(\Mman,\Nman)$ is given by $\EXP_{\Lman,\Mman,\Nman}(\Func)(t)(x) = \Func(t,x)$.
We will call this map an \myemph{exponential map} and often denote it simply by $\EXP$ whenever the sets $\Lman,\Mman,\Nman$ are understood from the context.

\begin{lemma}[Exponential law]\label{lm:exp_law}{\rm(cf.~\cite[Theorem 46.11]{Munkres:Top:2000})}
Let $\Lman,\Mman,\Nman$ be topological spaces and $\EXP:=\EXP_{\Lman,\Mman,\Nman}$ be the exponential map~\eqref{equ:exp_law_for_sets}.
Then the following statements hold.
\begin{enumerate}[leftmargin=*, label={\rm(\arabic*)}]
\item\label{enum:exp_law:1}
If $\Func:\Lman\times\Mman\to\Nman$ is continuous, then $\EXP(\Func):\Lman\to\Cont{\Mman}{\Nman}$ is continuous as well, that is
\[ \EXP\bigl(\Cont{\Lman\times\Mman}{\Nman}\bigr) \subset \Cont{\Lman}{\Cont{\Mman}{\Nman}}. \]

\item\label{enum:exp_law:2}
The induced map
\begin{equation}\label{equ:exp_law_general}
\EXP:\Cont{\Lman\times\Mman}{\Nman} \to \Cont{\Lman}{\Cont{\Mman}{\Nman}}
\end{equation}
is continuous with respect to the corresponding compact open topologies.

\item\label{enum:exp_law:3}
If $\Mman$ is \myemph{locally compact} then $\EXP\bigl(\Cont{\Lman\times\Mman}{\Nman}\bigr) = \Cont{\Lman}{\Cont{\Mman}{\Nman}}$, i.e. the map~\eqref{equ:exp_law_general} is a continuous bijection.

\item\label{enum:exp_law:4}
If, in addition to~\ref{enum:exp_law:3}, $\Lman$ is a \myemph{$T_3$-space} (again not necessarily Hausdorff), then~\eqref{equ:exp_law_general} is a homeomorphism.
\qed
\end{enumerate}
\end{lemma}
\begin{proof}
\ref{enum:exp_law:1}
Let $\Func:\Lman\times \Mman \to \Nman$ be a continuous map.
We need to prove that the map $\func = \EXP(\Func): \Lman \to \Cont{\Mman}{\Nman}$ given by $\func(t)(x) = \Func(t,x)$ is continuous.
Let $t\in\Lman$ and
\[
[\Bman,\Wman] = \{ \gfunc \in \Cont{\Mman}{\Nman} \mid \gfunc(\Bman)\subset\Wman\}
\]
be an open neighbourhood of $\func(t)$ in $\Cont{\Mman}{\Nman}$.
We need to find an open neighbourhood $\Uman$ of $t$ in $\Lman$ such that $\func(\Uman) \subset [\Bman,\Wman]$.

Notice that the assumption $\func(t)\in [\Bman,\Wman]$ means that $t\times\Bman \subset \Func^{-1}(\Wman)$.
Since $\Bman$ is compact, there exists an open neighbourhood $\Uman$ of $t$ in $\Lman$ such that $\Uman\times\Bman\subset\Func^{-1}(\Wman)$ as well.
In other words, $\func(s)(\Bman)\subset\Wman$ for each $s\in\Uman$, i.e. $\func(\Uman) \subset [\Bman,\Wman]$.

\smallskip

\ref{enum:exp_law:2}
Let $\Func\in\Cont{\Lman\times\Mman}{\Nman}$ and $\func = \EXP(\Func)\in\Cont{\Lman}{\Cont{\Mman}{\Nman}}$.
Let also $\Aman \subset \Lman$ and $\Bman \subset \Mman$ be compact subsets, $\Wman\subset\Nman$ an open set, and
\[
[\Aman, [\Bman,\Wman]] = \{
\gfunc:	\Lman \to \Cont{\Mman}{\Nman} \mid
\gfunc(\Aman)(\Bman) \subset \Wman
\}
\]
be an open subbase neighbourhood of $\func$ in $\Cont{\Lman}{\Cont{\Mman}{\Nman}}$.
Then the set $[\Aman\times\Bman, \Wman]$ is an open neighbourhood of $\Func$ in $\Cont{\Lman\times\Mman}{\Nman}$ such that
\[ \EXP\bigl(   [\Aman\times\Bman, \Wman] \bigr) \subset [\Aman, [\Bman,\Wman]]. \]
This means that $\EXP$ is continuous.

\smallskip

\ref{enum:exp_law:3}
Suppose $\Mman$ is locally compact and let $\func: \Lman \to \Cont{\Mman}{\Nman}$ be a continuous map.
We should prove that $\Func = \EXP^{-1}(\func): \Lman \times \Mman \to \Nman$, $\Func(t,x)=\func(t)(x)$, is continuous.
Let $(t,x)\in\Lman \times \Mman$ and $\Wman$ be an open neighbourhood of $\Func(t,x)$ in $\Nman$.
We need to find neighbourhoods $\Uman \subset\Lman$ and $\Vman \subset\Mman$ of $t$ and $x$ respectively such that $\Func(\Uman\times\Vman)\subset\Wman$.

As $\func(t):\Mman\to\Nman$ is continuous, $\func(t)^{-1}(\Wman)$ is an open neighbourhood of $x$ in $\Mman$, whence there is a compact $\Bman\subset\Mman$ such that $x\in\Int{\Bman} \subset\Bman \subset \Uman$.
Then $[\Bman,\Wman]$ is an open neighbourhood of $\func(t)$, and since $\func:\Lman\to\Cont{\Mman}{\Nman}$ is continuous at $t$, there exists an open neighbourhood $\Uman \subset\Lman$ of $t$ such that $\func(\Uman)\subset [\Bman,\Wman]$.
This implies that $\Func(\Uman\times\Bman) \subset \Wman$.
In particular, one can put $\Vman=\Int{\Bman}$, and then we will have that $\Func(\Uman\times\Vman) \subset \Wman$.

\smallskip

\ref{enum:exp_law:4}
Suppose that $\Mman$ is locally compact and $\Lman$ is a $T_3$ space, i.e.\! for each point $x\in\Lman$ and its open neighbourhood $\Uman$ there exists an open neighbourhood $\Zman$ such that $x\in\Zman\subset\overline{\Zman} \subset \Uman$.
Let $\Func\in \Cont{\Lman\times\Mman}{\Nman}$, $\func = \EXP(\Func)$, and $[\Cman,\Wman]$ be an open neighbourhood of $\Func$, where $\Cman\subset\Lman\times\Mman$ is a compact subset and $\Wman\subset\Nman$ is open.
We need to prove that $\EXP([\Cman,\Wman])$ contains an open neighbourhood of $\func$ in $\Cont{\Lman}{\Cont{\Mman}{\Nman}}$.

Since $\Cman$ is compact and $\Func^{-1}(\Wman)$ is its open neighbourhood in $\Lman\times\Mman$, there exist finitely many open sets $\Uman_1,\ldots,\Uman_k \subset \Lman$, and $\Vman_1,\ldots,\Vman_k \subset \Mman$, such that
\[
\Cman \ \subset \ \mathop{\cup}\limits_{i=1}^{k} \Uman_i \times \Vman_i \ \subset \ \Func^{-1}(\Wman).
\]
Moreover, since $\Lman$ is $T_3$ and $\Mman$ is locally compact, one can assume in addition that
\begin{enumerate}[label={\rm(\alph*)}]
\item
each $\Vman_i = \Int{\Bman_i}$ for some compact $\Bman_i \subset\Mman$ and that
\item
$
	\Cman \ \subset\ \mathop{\cup}\limits_{i=1}^{k} \Uman_i \times \Vman_i
	\ \subset \
	\mathop{\cup}\limits_{i=1}^{k} \overline{\Uman_i} \times \Bman_i
	\ \subset \
	\Func^{-1}(\Wman).
$
\end{enumerate}
Let $p:\Lman\times\Mman\to\Lman$, $p(t,x)=t$, be the natural projection.
Then $p(\Cman)$ is compact.
Hence $\Aman_i := \overline{\Uman_i} \cap p(\Cman)$ is a closed subset of $p(\Cman)$ and therefore compact.
It then follows that
\[ \Cman \ \subset \
\mathop{\cup}\limits_{i=1}^{k} \Aman_i \times \Bman_i  \ \subset \ \Func^{-1}(\Wman),
\]
which also implies that $\mathop{\cup}\limits_{i=1}^{k} [\Aman_i, [\Bman_i,\Wman]] \ \subset \  \EXP\bigl( [\Cman,\Wman] )$.
\end{proof}

\subsection*{Loop spaces}
\newcommand\Loop[2]{\Omega(#1,#2)}
\newcommand\LoopJ[3]{\Omega_{#1}(#2,#3)}
\newcommand\LoopMx{\Loop{\Mman}{\ptm}}
\newcommand\LoopJMx{\LoopJ{\Jman}{\Mman}{\ptm}}

We will prove one simple lemma which was used in Theorem~\ref{th:loop_space}.
Let $\Mman$ be a topological space, $\ptm\in\Mman$ a point, $\Iman=[0,1]$, $\Jman = [0,0.1] \cup [0.9,1]$, and
\begin{align*}
\LoopMx &:= 
\CrmPairs{0}{\Iman}{\partial\Iman}{\Mman}{\ptm} = \{ \gamma\in\Cont{\Iman}{\Mman} \mid \gamma(0)=\gamma(1)=x \}, \\
%%%%%%%%%%%%%%%%%%%%%%
\LoopJMx &:= 
\CrmPairs{0}{\Iman}{\Jman}{\Mman}{\ptm} = \{ \gamma\in\Cont{\Iman}{\Mman} \mid \gamma(t)=x \ \text{for all} \ t\in\Jman \},
%\Cont{(\Iman,\Jman)}{(\Mman, \ptm)},
\end{align*}
be ``loop spaces'' of $\Mman$ at $\ptm$.

\begin{lemma}\label{lm:loop_inclusion}
The inclusion $\LoopJMx \subset \LoopMx$ is a homotopy equivalence.
\end{lemma}
\begin{proof}
Fix any continuous function $\eta:[0,1]\to[0,1]$ such that $\eta[0,0.1] = 0$ and $\eta[0.9,1] = 1$, and define the following homotopy:
\begin{align*}
&H:\Iman \times  \Cont{\Iman}{\Mman} \to \Cont{\Iman}{\Mman}, &
%%%%%%%%%%%
H(\ptt,\gamma)(\ptq) &= \gamma\bigl((1-\ptt)\ptq + \ptt\eta(\ptq)\bigr).
\end{align*}
This map is continuous as the following composition:
\begin{multline*}
H :
\Iman \times  \Cont{\Iman}{\Mman} 
\xrightarrow{(\ptt, \gamma) \ \mapsto \ ( (1-\ptt)\id_{\Iman} + \ptt\eta, \ \gamma )} \\ 
%%%%%%%%%%%%%%%%%
\longrightarrow
\Cont{\Iman}{\Iman} \times \Cont{\Iman}{\Mman} 
\xrightarrow[Lemma~\ref{lm:cotop}\ref{enum:lm:cotop::composition}]{(\psi, \gamma) \ \mapsto \ \gamma\circ\psi}
%%%%%%%%%%%%%%%%%
 \Cont{\Iman}{\Mman}.
\end{multline*}
One easily checks that
\begin{enumerate}[label={\rm(\alph*)}]
\item $H_{0} = \id_{\Cont{\Iman}{\Mman}}$;
\item $H\bigl([0,1]\times\LoopMx\bigr) \subset \LoopMx$ and $H\bigl([0,1]\times\LoopJMx\bigr) \subset \LoopJMx$;
\item $H_{1}\bigl(\LoopMx\bigr) \subset \LoopJMx$.
\end{enumerate}
These properties mean that the restriction of $H:\Iman \times \LoopMx\to \LoopMx$ is a \myemph{deformation} of $\LoopMx$ into $\LoopJMx$, which implies that the inclusion $\LoopJMx \subset  \LoopMx$ is a homotopy equivalence.
\end{proof}

\newcommand\ball[3]{B^{#1}_{#2}(#3)}
\newcommand\cone[2]{\bR^{#1}_{+}} %{(\bR^{#1})_{#2}^{+}}
\newcommand\ind[1]{\mathrm{ind}(#1)}

\section{Whitney topologies}\label{sect:whitney_topologies}
In this section we recall basic results on Whitney topologies, and refer the reader for details to~\cite[Chapter~2]{Hirsch:DiffTop}, \cite[Chapter 4]{Michor:SMS:1980}, or~\cite[Chapter~8]{Mukherjee:DT:2015}, and to~\cite[Chapter~9]{MargalefOuterelo:DT:1992} for manifolds with corners.

\subsection*{Manifolds with corners}
The following subset of $\bR^{m}$:
\[ 
    \cone{m}{\Lambda} = \{ (\ptm_1,\ldots,\ptm_{m})\in\bR^m \mid \ptm_i\geq0, \ i=1,\ldots,m \}
\]
will be called a \myemph{non-negative quadrant} or simply a \myemph{quadrant}.
A \myemph{manifold with corners} is a Hausdorff topological space $\Mman$ having countable base and such that for every point $\ptm\in\Mman$ there exists a neighbourhood $\Uman$ and an open embedding $\phi:\Uman \to \cone{m}{\Lambda}$ into some quadrant.
Such a triple $(\Uman,\phi,\Lambda)$ is called a \myemph{chart} of $\Mman$.
Let $\phi(\ptm)=(\ptm_1,\ldots,\ptm_{m}) \in \cone{m}{\Lambda}$.
Then the number of zero coordinates of $\phi(\ptm)$ is called the \myemph{index} of $\ptm$ with respect to the given chart and denoted by $\ind{\ptm}$.
In particular, $\ptm$ is an internal point iff $\ind{\ptm}=0$.

Then one can define in a usual way a notion of a $\Cr{\rx}$ atlas and a differentiable structure of class $\Cr{\rx}$ on a manifold with corners for $0\leq \rx\leq\infty$, see for details~\cite{Douady:SHC:1961} or \cite[Chapter~1]{MargalefOuterelo:DT:1992}.
If $\rx=0$, then a $\Cr{0}$ manifold with corners is the same as a topological manifold with boundary.
However, if $\rx>0$, then the indices of points are invariants of the corresponding $\Cr{\rx}$ structures, and one can define the following subsets 
$\partial^{k}\Mman = \{\ptm\in\Mman \mid \ind{\ptm}= k \}$ for $k=0,1,\ldots,m$, which are invariant with respect to $\Cr{1}$ diffeomorphisms of $\Mman$.
For instance, $\Int{\Mman} = \partial^{0}\Mman$, $\partial\Mman = \cup_{k=1}^{m} \partial^{k}\Mman$, and a \myemph{manifold with boundary} is the same as manifold with corners such that every point of $\Mman$ has index $\leq 1$. 

Due to Whithey extension theorem it is possible to prove all standard results on manifolds with boundary for manifolds with corners, see \cite{MargalefOuterelo:DT:1992}.
%\subsection*{Manifolds with corners}
%Let $\Lambda = \{ \lambda_i:\bR^m \to \bR \}_{i=1}^{k}$ be a finite (possibly empty) linearly independent family of linear functions.
%Then the set 
%\[ \cone{m}{\Lambda} = \{ \ptm\in\bR^m \mid  \lambda_i(x)\geq0 \ \text{for all} \ i=1,\ldots,k \}\]
%will be called the \myemph{non-negative cone} of $\Lambda$.
%A \myemph{manifold with corners} is a Hausdorff topological space $\Mman$ having countable base and such that for every point $\ptm\in\Mman$ there exists a neighbourhood $\Uman$ and an open embedding $\phi:\Uman \to \cone{m}{\Lambda}$ into some cone.
%Such a triple $(\Uman,\phi,\Lambda)$ is called a \myemph{chart} of $\Mman$.
%Then one can define in a standard way a notion of a $\Cr{\rx}$ atlas and a differentiable structure of class $\Cr{\rx}$ on a manifold with corners for $0\leq \rx\leq\infty$, see for details~\cite[Chapter~1]{MargalefOuterelo:DT:1992}.

\subsection*{Jet bundles}
Let $\Mman$ and $\Nman$ be smooth manifolds (possibly with corners) of dimensions $m$ and $n$ respectively.
If $\phi:\Vman\to\cone{m}{\Lambda}$, $\psi:\Wman\to\cone{n}{\Lambda'}$ be two charts from the atlases of $\Mman$ and $\Nman$ respectively.
Let also $\dif:\Mman\to\Nman$ be a map such that $\dif(\Uman) \subset\Vman$.
Then a \myemph{local presentation} of $\dif$ in these charts is the following map
\begin{equation}\label{equ:local_presentation}
	\cone{m}{\Lambda} \supset \phi(\Uman) \ \xrightarrow{~\psi\circ\dif\circ\phi^{-1}~} \ \psi(\Vman) \subset \cone{n}{\Lambda'}.
\end{equation}

Let $0\leq \rx\leq\infty$, $\ptm\in\Mman$, and $\func,\gfunc:(\Mman,\ptm)\to\Nman$ be two germs of $\Cr{\rx}$ maps at $\ptm$.
Say that $\func$ and $\gfunc$ are \myemph{$\rx$-equivalent at $\ptm$} if in some local presentation~\eqref{equ:local_presentation} of $\func$ in the local chart containing $\ptm$ the Taylor polynomials of the corresponding coordinate functions of $\func$ and $\gfunc$ at $\ptm$ up to order $\rx$ coincide.
This, in particular, means that $\func(\ptm)=\gfunc(\ptm)$.

The $\rx$-equivalence class at $\ptm$ of a germ $\func$ is called the \myemph{$\rx$-jet of $\func$ at $\ptm$} and is denoted by $\jr{\rx}{\func}{\ptm}$, the set of such classes is denoted by $\Jr{\rx}{\Mman}{\Nman}_{\ptm}$, and the following union
\[\Jr{\rx}{\Mman}{\Nman}:= \mathop{\cup}\limits_{\ptm\in\Mman} \Jr{\rx}{\Mman}{\Nman}_{\ptm}\]
is called the \myemph{space of $\rx$-jets of maps $\Mman\to\Nman$}.
In particular, it follows from the definition that there is a natural identification $\Jr{0}{\Mman}{\Nman}=\Mman\times\Nman$.

It is known that for $r<\infty$ the set $\Jr{\rx}{\Mman}{\Nman}$ has a structure of a smooth finite dimensional manifold%
\footnote{\,\cite[Proposition~9.1.7]{MargalefOuterelo:DT:1992} for manifolds with corners.},
and we have a sequence of vector bundles:
\begin{multline*}
\cdots \xrightarrow{\jpr{\rx+1}} \Jr{\rx}{\Mman}{\Nman}  \xrightarrow{\jpr{\rx}} \Jr{r-1}{\Mman}{\Nman} \to \cdots \\ 
\cdots \xrightarrow{\jpr{2}} \Jr{1}{\Mman}{\Nman} \xrightarrow{\jpr{1}} \Jr{0}{\Mman}{\Nman}\equiv \Mman\times\Nman,
\end{multline*}
such that the fibres of $\jpr{i}$ can be regarded as the spaces of homogeneous polynomial maps $\bR^m \to \bR^n$ of degree $i$.
In particular, the composition $\pr{\rx}=\jpr{1}\circ\cdots\circ\jpr{\rx}:\Jr{\rx}{\Mman}{\Nman} \to \Mman\times\Nman$ is also a vector bundle, whose fibres are the spaces of polynomial maps $\bR^m \to \bR^n$ of degree $i$.

Let $\pi_{\Mman}:\Mman\times\Nman \to\Mman$ and $\pi_{\Nman}:\Mman\times\Nman \to\Nman$ be the natural projections.
Then we get the following two locally trivial fibrations:
\begin{align*}
\src{\rx}  &= \pi_{\Mman} \circ \pr{\rx}:\Jr{\rx}{\Mman}{\Nman} \to \Mman, &
\ \ 
\dest{\rx} &= \pi_{\Nman} \circ \pr{\rx}:\Jr{\rx}{\Mman}{\Nman} \to \Nman,
\end{align*}
called respectively the \myemph{source} and \myemph{destination} maps.
In particular, for $\rx$-jet $\sigma= \jr{\rx}{\func}{\ptl}\in\Jr{\rx}{\Mman}{\Nman}$ represented by a germ $\func:(\Mman,\ptm) \to (\Nman, \ptn)$, we have that $\src{\rx}(\sigma) = \ptm$, and $\dest{\rx}(\sigma) = \ptn = \func(\ptm)$.

Further recall that for every $\dif\in\Crm{\rx}{\Mman}{\Nman}$ one can define its \myemph{$\rx$-th jet prolongation}
\[
	\jr{\rx}{\dif}{}:\Mman \to\Jr{\rx}{\Mman}{\Nman}
\]
being a section of $\src{\rx}$.
In particular, $\jr{0}{\dif}{\ptm} = (\ptm,\dif(\ptm))$, so $\jr{0}{\dif}{}$ is just the embedding of $\Mman$ into $\Mman\times\Nman$ as the graph of $\dif$.
More generally, if we fix some local presentation of $\dif$ as a map $\dif=(\dif_1,\ldots,\dif_n):\cone{m}{\Lambda} \supset \Vman \to \bR^n$ from some open subset $\Vman$ of $\cone{m}{\Lambda}$, then $\jr{\rx}{\dif}{}$ has the following local presentation $\jr{\rx}{\dif}{}:\Uman\to\Uman \times \bR^{m+(1+m + m^2 +\cdots + m^{\rx})}$ defined by
\[
	\jr{\rx}{\dif}{\ptm_1,\ldots,\ptm_m} =
		\left(\ptm_1,\ldots,\ptm_m,
			  \left\{
					  \tfrac{\partial^i\dif_j}{\partial \ptm_{a_1}\cdots\partial \ptm_{a_i}}
			  \right\}_{a_1,\ldots,a_i=1,\ldots,m}^{ i=0,\ldots,r, \ j=1,\ldots,n}
		\right).
\]
Its coordinate functions are coordinates of the point $\ptm=(\ptm_1,\ldots,\ptm_m)\in\Uman$ and all partial derivatives up to order $\rx$ of all coordinate functions of $\dif$.

Also notice that the above correspondence $\dif\mapsto\jr{i}{\dif}{}$, $r<\infty$, gives an \myemph{injection}
\begin{equation}\label{equ:r-jet_ext}
	\ejr{\rx}:\Crm{\rx}{\Mman}{\Nman} \subset \Crm{0}{\Mman}{\Jr{r}{\Mman}{\Nman}},
\end{equation}
and thus one can identify $\Crm{\rx}{\Mman}{\Nman}$ with its image in $\Crm{0}{\Mman}{\Jr{r}{\Mman}{\Nman}}$.
By~\cite[Chapter~2, Theorem~4.3]{Hirsch:DiffTop}, $\ejr{\rx}\bigl( \Crm{\rx}{\Mman}{\Nman} \bigr)$ is a closed subset of the space $\Crm{0}{\Mman}{\Jr{\rx}{\Mman}{\Nman}}$.

\begin{definition}\label{def:W0}
Let $0\leq \rx \leq \infty$.
Then the \myemph{weak topology $\Wr{0}$} on $\Crm{0}{\Mman}{\Nman}$ is just the compact open topology.

The \myemph{weak topology $\Wr{\rx}$} on $\Crm{\rx}{\Mman}{\Nman}$ for $0<\rx<\infty$ is the initial topology with respect to the inclusion map $\ejr{\rx}$, see~\eqref{equ:r-jet_ext}.

The \myemph{weak topology $\Wr{\infty}$} on $\Crm{\infty}{\Mman}{\Nman}$ is the topology generated by all topologies $\Wr{\rx}$ for $0\leq \rx < \infty$.

Finally, the \myemph{weak topology $\Wr{i}$} on $\Crm{\rx}{\Mman}{\Nman}$ for $i<\rx$ is the topology induced from the inclusion $\Crm{\rx}{\Mman}{\Nman} \subset \Crm{i}{\Mman}{\Nman}$. 
\end{definition}

The space $\Crm{\rx}{\Mman}{\Nman}$ with the topology $\Wr{i}$, $0\leq i\leq \rx$, will be denoted by $\WWW{\rx}{i}{\Mman}{\Nman}$.
If $i=\rx$, then we omit the index for the topology and thus write $\Crm{\rx}{\Mman}{\Nman}$ instead of $\WWW{\rx}{\rx}{\Mman}{\Nman}$.

Since $\Jr{i}{\Mman}{\Nman}$ is a $T_3$-space (in fact even a Hausdorff finite-dimensional manifold), it follows from Lemma~\ref{lm:cotop}\ref{enum:lm:cotop::subbase_prop:closure_KU} that $\Crm{0}{\Mman}{\Jr{i}{\Mman}{\Nman}}$ and therefore its subspace $\WWW{\rx}{i}{\Mman}{\Nman}$ is a $T_3$ space as well.
Moreover, as $\Wr{\infty}$ is generated by $T_3$ topologies $\Wr{\rx}$, $0\leq\rx<\infty$, one easily checks that $\Wr{\infty}$ is also a $T_3$ topology.

\begin{lemma}[Composition of jets]\label{lm:jets_composition}{\rm(\cite[Chapter~1, 1.4]{Michor:SMS:1980}, \cite[Proposition~8.3.2]{Mukherjee:DT:2015}, or~\cite[Lemma~9.6.2]{MargalefOuterelo:DT:1992} for manifolds with corners)}
Let $\Lman,\Mman,\Nman$ be smooth manifolds and $\rx<\infty$.
Then the following set
\begin{multline*}
\Jr{\rx}{\Lman}{\Mman} \times_{\Mman} \Jr{\rx}{\Mman}{\Nman} = \\ =
 \{
(\sigma, \tau) \in \Jr{\rx}{\Lman}{\Mman} \times \Jr{\rx}{\Mman}{\Nman} \mid
\dest{\rx}(\sigma) = \src{\rx}(\tau)
\}
\end{multline*}
is a real analytic submanifold of $\Jr{\rx}{\Lman}{\Mman} \times \Jr{\rx}{\Mman}{\Nman}$.
Moreover, the ``\myemph{composition of $\rx$-jets}'' map $\nu: \Jr{\rx}{\Lman}{\Mman} \times_{\Mman} \Jr{\rx}{\Mman}{\Nman} \to \Jr{\rx}{\Lman}{\Nman}$ 
\begin{gather*}
%\nu: \Jr{\rx}{\Lman}{\Mman} \times_{\Mman} \Jr{\rx}{\Mman}{\Nman} \to \Jr{\rx}{\Lman}{\Nman}, \\
\nu\bigl(
		\,\jr{\rx}{\func}{\ptm}, \, \jr{\rx}{\gfunc}{\func(\ptm)} \,
   \bigr) \,=\,
   \jr{\rx}{\gfunc\circ\func}{\ptm},
\end{gather*}
is continuous, in fact, even real analytic.
\end{lemma}

\begin{corollary}\label{cor:Wr_cont_composition}
{\rm(\cite[Chapter 2, \S4, Exercise~10]{Hirsch:DiffTop})}
The composition map
\[
\mult_{\rx}: \Crm{\rx}{\Lman}{\Mman} \times \Crm{\rx}{\Mman}{\Nman} \to \Crm{\rx}{\Lman}{\Nman}, \qquad
\mult_{\rx}(\func,\gfunc) = \gfunc\circ\func,
\]
is continuous.
\end{corollary}
\begin{proof}
Such a statement is usually formulated and proved for strong topologies and $\Crm{\rx}{\Lman}{\Mman}$ is replaced with the space of proper maps $\mathrm{Prop}^{\rx}(\Lman,\Mman)$, see e.g.~\cite[Proposition~8.3.4]{Mukherjee:DT:2015}.
We will present a ``formal'' proof of continuity of $\mult_{\rx}$ just to illustrate several typical formal arguments based on Lemma~\ref{lm:jets_composition} and elementary properties of compact open topologies described in Lemma~\ref{lm:cotop}.
Such kind of arguments will appear several times later in the paper.

Due to the definition of topology $\Wr{\rx}$ it suffices to show that the map
\[ \Crm{\rx}{\Lman}{\Mman} \times \Crm{\rx}{\Mman}{\Nman} \to \Cont{\Lman}{\Jr{\rx}{\Lman}{\Nman}},
\qquad
\mult_{\rx}'(\func,\gfunc)(\ptl) = \jr{r}{\gfunc\circ\func}{\ptl},
\]
is continuous.
Consider the following composition of continuous maps:
\begin{align*}
&\alpha: \ \Crm{\rx}{\Lman}{\Mman} \times \Crm{\rx}{\Mman}{\Nman}
\xrightarrow{(\func,\gfunc) \ \mapsto \ (\jr{\rx}{\func}{},\ \func, \ \jr{\rx}{\gfunc}{})} \\
% %%%%%%%%%%%%%%%%%%%%%%
&\quad\to\
\Cont{\Lman}{ \Jr{\rx}{\Lman}{\Mman} }  \times
\Cont{\Lman}{\Mman}  \times
\Cont{\Mman}{ \Jr{\rx}{\Mman}{\Nman} }
\xrightarrow[\text{Lemma~\ref{lm:cotop}\ref{enum:lm:cotop::composition}}]{(\sigma,\func,\tau) \ \mapsto \ (\sigma, \ \tau\circ\func)} \\
%%%%%%%%%%%%%%%%%
&\quad\to \
\Cont{\Lman}{ \Jr{\rx}{\Lman}{\Mman} }  \times
\Cont{\Lman}{ \Jr{\rx}{\Mman}{\Nman} } \xrightarrow[\text{Lemma~\ref{lm:cotop}\ref{enum:lm:cotop:map_into_product}}]{\cong} \\
% %%%%%%%%%%%%%%%%%
&\quad\to \
\Cont{\Lman}{ \Jr{\rx}{\Lman}{\Mman} \times \Jr{\rx}{\Mman}{\Nman} }.
% %%%%%%%%%%%%%%%%%
\end{align*}
Notice that
\begin{multline*}
	\alpha(\func,\gfunc)(\ptl) = \bigl( \jr{\rx}{\func}{\ptl},  \jr{\rx}{\gfunc}{\func(\ptl)}) \in \Jr{\rx}{\Lman}{\Mman} \times_{\Mman} \Jr{\rx}{\Mman}{\Nman} \ \subset \\
	\subset \ \Jr{\rx}{\Lman}{\Mman} \times \Jr{\rx}{\Mman}{\Nman},
\end{multline*}
and $\nu\circ\alpha(\func,\gfunc) = \jr{\rx}{\gfunc\circ\func}{} = \mult'(\nu\circ\alpha)$.
Therefore $\alpha$ can be regarded as a continuous map
\[
\alpha: \Crm{\rx}{\Lman}{\Mman} \times \Crm{\rx}{\Mman}{\Nman}  \to
\Cont{\Lman}{\Jr{\rx}{\Lman}{\Mman} \times_{\Mman} \Jr{\rx}{\Mman}{\Nman}},
\]
so that
\begin{multline*}
\mult_{\rx}' = \indmap{\nu}\circ\alpha:
\Crm{\rx}{\Lman}{\Mman} \times \Crm{\rx}{\Mman}{\Nman}
	\xrightarrow{\alpha} \\ \to
\Cont{\Lman}{\Jr{\rx}{\Lman}{\Mman} \times_{\Mman} \Jr{\rx}{\Mman}{\Nman}}
	\xrightarrow{\indmap{\nu}} \Cont{\Lman}{\Jr{\rx}{\Lman}{\Nman}}.
\end{multline*}
Hence $\mult_{\rx}'$ and therefore $\mult_{\rx}$ are continuous for $\rx<\infty$.

Let $\rx=\infty$.
Then for each $\sx<\infty$ the map
\[
	\mult_{\infty}:\Ci{\Lman}{\Mman} \times \Ci{\Mman}{\Nman} \xrightarrow{\id}
	\WWW{\infty}{\sx}{\Lman}{\Mman}  \times \WWW{\infty}{\sx}{\Mman}{\Nman}
	\xrightarrow{\mult_{\sx}}
	\WWW{\infty}{\sx}{\Lman}{\Mman}
\]
is continuous into the topology $\Wr{\sx}$ of $\Ci{\Lman}{\Nman}$.
Since the topology $\Wr{\infty}$ on $\Ci{\Lman}{\Nman}$ is generated by all the topologies $\Wr{\sx}$ for $\sx < \infty$, we obtain that $\mult_{\infty}$ is also continuous into the topology $\Wr{\infty}$ of $\Ci{\Lman}{\Nman}$.
\end{proof}

\section{Jets of maps from direct products}\label{sect:maps_from_prods}
Let $\Lman$ be a smooth manifold (possibly with corners) of dimension $l$.
We will now describe in terms of jets partial derivatives along $\Mman$ of $\Cr{\rx}$ maps $\Lman\times\Mman \to\Nman$.

Let $(\ptl,\ptm)\in\Lman\times\Mman$, and $\func,\gfunc:\Lman\times\Mman \to\Nman$ be two germs of maps at $(\ptl,\ptm)$.
Say that $\func$ and $\gfunc$ are \myemph{$\rx$-equivalent along $\Mman$ at $(\ptl,\ptm)$} whenever the restriction maps $\func,\gfunc:\ptl\times\Mman\to\Nman$ are $\Cr{\rx}$ and those maps are also $\rx$-equivalent at $x$.
The corresponding equivalence class will be denoted by $\jrm{\rx}{\Mman}{\func}(\ptl,\ptm) := \jr{\rx}{\func|_{\ptl\times\Mman}}{\ptm}$.

Evidently, if two germs $\func,\gfunc:\Lman\times\Mman \to\Nman$ are $\rx$-equivalent at a point $(\ptl,\ptm)$, then they also $\rx$-equivalent along $\Mman$.
This means that there is a natural $\Cinfty$ map
\[
	\xi: \Jr{\rx}{\Lman \times\Mman}{\Nman} \to \Lman \times \Jr{\rx}{\Mman}{\Nman},
	\qquad
	\xi\bigl(\jr{\rx}{\func}{\ptl,\ptm}\bigr) = \bigl(\ptl, \jrm{\rx}{\Mman}{\func}(\ptl,\ptm)\bigr).
\]

Notice that there is also the inverse $\Cinfty$ map
\[
	\sigma: \Lman \times \Jr{\rx}{\Mman}{\Nman} \to \Jr{\rx}{\Lman \times\Mman}{\Nman},
	\qquad
	\sigma\bigl(\ptl, \jrm{\rx}{\Mman}{\func}(\ptl,\ptm)\bigr) = \jr{\rx}{\func\circ p_{\Mman}}{\ptl,\ptm},
\]
where $p_{\Mman}:\Lman\times\Mman \to \Mman$ is the natural projection.
Then $\sigma\circ\xi=\id_{\Lman \times \Jr{\rx}{\Mman}{\Nman}}$, so $\Lman \times \Jr{\rx}{\Mman}{\Nman}$ can be viewed as a retract of $\Jr{\rx}{\Lman \times\Mman}{\Nman}$.

Moreover, we also have the following commutative diagram:
\begin{equation}\label{equ:jets_along_M}
\begin{gathered}
\xymatrix@C=1em{
\Lman \times \Jr{\rx}{\Mman}{\Nman}
	\ar@{^(->}[rr]^{\sigma}
	\ar[rrd]_-{\id_{\Lman} \times \src{\rx}}  &&
\Jr{\rx}{\Lman \times\Mman}{\Nman}
    \ar@{->>}[rr]^{\xi}
	\ar[d]^-{\src{\rx}'}  &&
\Lman \times \Jr{\rx}{\Mman}{\Nman}
	\ar[lld]^-{\id_{\Lman} \times \src{\rx}} \\
%%%%%%%%%%
&&\Lman\times\Mman
}
\end{gathered}
\end{equation}
where $\src{\rx}:\Jr{\rx}{\Mman}{\Nman}\to\Mman$ and $\src{\rx}':\Jr{\rx}{\Lman\times\Mman}{\Nman}\to\Lman\times\Mman$ are the corresponding $\rx$-jet bundles.
In particular, all down arrows are fibre bundles.

By definition, the topology $\Wr{\rx}$ on $\Crm{\rx}{\Lman\times\Mman}{\Nman}$ is obtained from the identification of that space with a subset of the space of sections of the middle arrow in~\eqref{equ:jets_along_M} by the $\rx$-jet prolongation map~\eqref{equ:r-jet_ext}.
We will need to consider continuous maps $\Lman\times\Mman\to\Nman$ which induce sections of the left arrow in~\eqref{equ:jets_along_M}.
Such maps are studied in many papers and in different variations, e.g.~\cite{Amann:ODE:1990, Glockner:JFA:2002, AlzaareerSchmeding:EM:2015}.
Usually they are defined only for maps into locally convex vector spaces in order to develop differential calculus in those space.
For our purposes we need a ``nonlinear'' variant for maps into a manifold.

\begin{definition}\label{def:part_Cr}
{\rm(e.g.~\cite[II.6, page 91]{Amann:ODE:1990}, \cite[\S1.4]{Glockner:JFA:2002}, \cite[Definition 3.1]{AlzaareerSchmeding:EM:2015}).}
Let $\Lman$ be a topological space, $\Uman\subset\Lman \times \Mman$ an open subset, and $0\leq \rx < \infty$.
A continuous map $\func:\Uman \to \Nman$ is \myemph{$\Cr{0,\rx}$ along $\Mman$}, or simply $\Cr{0,\rx}$ (if the decomposition of the topological space $\Lman\times\Mman$ into a direct product is understood from the context), whenever
\begin{enumerate}[leftmargin=5ex, label={\rm(\roman*)}]
\item\label{enum:ft:ft_Cr}
for each $\ptl\in\Lman$ the restriction $\func|_{\Uman\cap(\ptl\times\Mman)}:\Uman\cap(\ptl\times\Mman) \to \Nman$ is $\Cr{\rx}$;
\item\label{enum:ft:df_Cs}
the map $\jetM{\rx}{\Mman}: \Uman \to \Jr{\rx}{\Mman}{\Nman}$ defined by $\jrm{\rx}{\Mman}{\func}(\ptl,\ptm) = \jr{\rx}{\func|_{\ptl\times\Mman}}{\ptm}$ is continuous;
\end{enumerate}
That map $\func:\Uman \to \Nman$ is $\Cr{0,\infty}$ (along $\Mman$) if it is $\Cr{0,\rx}$ along $\Mman$ for all $\rx<\infty$.
Denote by $\Crm{0,\rx}{\Uman}{\Nman}$, $0\leq\rx\leq\infty$, the subset of $\Crm{0}{\Lman\times\Mman}{\Nman}$ consisting of $\Cr{0,\rx}$ maps.
\end{definition}

\begin{remark}\rm
Let $(\ptl,\ptm)\in\Lman \times \Mman$.
Choose local coordinates local coordinates $(\ptm_1,\ldots,\ptm_m)$ in $\Mman$ near $x$, and $(y_1,\ldots,y_n)$ in $\Nman$ near $\func(\ptl,\ptm)$, so $\func$ can be regarded as a germ of a map $\func=(\func_1,\ldots,\func_n):\Lman\times\cone{m}{\Lambda}\to\cone{n}{\Lambda'}$.
Then condition~\ref{enum:ft:df_Cs} means that $\func$ and each partial derivative up to order $\rx$ in $(\ptm_1,\ldots,\ptm_m)$ of each coordinate function of $\func$ is continuous in $(\ptl,\ptm)$.
However, if $\Lman$ is a smooth manifold, then $\func$ is not necessarily even differentiable in $\Lman$.
\end{remark}

\begin{example}\label{exmp:C0r_func}\rm
Consider the function
\[ \func:\bR\times\bR\to\bR, \qquad \func(\ptl,\ptm) = |\ptl|\cos(\ptl \ptm + \ptm^2).\]
Then $\func$ is $\Cr{0,\infty}$, since it is $\Cinfty$ in $\ptm$ and each of its partial derivatives in $\ptm$ is continuous in $(\ptl,\ptm)$.
On the other hand, $\func$ is only continuous but not even $\Cr{1}$ in $(\ptl,\ptm)$ due to the presence of the multiple $|\ptl|$ is the formula for $\func$.
\end{example}

The following simple statement is not explicitly formulated in the above works~\cite{Amann:ODE:1990, Glockner:JFA:2002, AlzaareerSchmeding:EM:2015}, though it gives a very important and useful characterization of $\Cr{0,\rx}$ property in terms of continuity of adjoint maps into compact open topologies.

\begin{lemma}[Characterization of $\Cr{0,\rx}$ maps]\label{lm:exp_law_finite}
Let $0\leq \rx < \infty$, $\Lman$ be a topological space, and $\func:\Lman\to\Crm{\rx}{\Mman}{\Nman}$ be a map.
Then the following conditions are equivalent.
\begin{enumerate}[leftmargin=*, label={\rm(\arabic*)}]
\item\label{enum:exp_law:f}
$\func:\Lman\to\Crm{\rx}{\Mman}{\Nman}$ is continuous into the topology $\Wr{\rx}$ of $\Crm{\rx}{\Mman}{\Nman}$.

\item\label{enum:exp_law:ejf}
The composition map
\[
\ejr{\rx}\circ \func:
    \Lman
	    \ \xrightarrow{~\func~} \
	\Crm{\rx}{\Mman}{\Nman}
		\ \xrightarrow{~\ejr{\rx}~} \
	\Crm{0}{\Mman}{\Jr{\rx}{\Mman}{\Nman}}
\]
is continuous into the topology $\Wr{0}$ of $\Crm{0}{\Mman}{\Jr{\rx}{\Mman}{\Nman}}$.

\item\label{enum:exp_law:F}
The map $\hat{\Func}:\Lman\times\Mman\to\Jr{\rx}{\Mman}{\Nman}$ defined by $\hat{\Func}(\ptl,\ptm) = j^r(f(\ptl))(x)$ is continuous as a map from the standard product topology of $\Lman\times\Mman$.

\item\label{enum:exp_law:loc_coords}
Let $\Func =( \Func_1,\ldots,\Func_n): \Lman\times\cone{m}{\Lambda} \ \supset \ \Lman\times \Vman \to \cone{n}{\Lambda'}$ be any local presentation of the map $F:\Lman\times\Mman \to\Nman$, $\Func(\ptl,\ptm) = \func(\ptl)(\ptm)$, where $\Vman\subset\cone{m}{\Lambda}$ is an open set.
Then each partial derivative
\[
\frac{\partial^i\Func_j}{\partial\ptm_{a_1}\cdots\partial\ptm_{a_i}}:\Lman \times \Vman \to \bR
\]
in $(\ptm_1,\ldots,\ptm_m)\in\Vman$ up to order $\rx$ of each coordinate function $\Func_i$ of $\Func$ is a continuous function on $\Lman \times \Vman$.

\item\label{enum:exp_law:FT}
The adjoint map $F:\Lman\times\Mman \to\Nman$, $\Func(s,x) = \func(s)(x)$, is $\Cr{0,\rx}$.
\end{enumerate}
\end{lemma}
\begin{proof}
The equivalences \ref{enum:exp_law:f}$\Leftrightarrow$\ref{enum:exp_law:ejf} and \ref{enum:exp_law:F}$\Leftrightarrow$\ref{enum:exp_law:FT} follow from the definitions of the topology $\Wr{\rx}$ and property $\Cr{0,\rx}$ respectively.

Also by Lemma~\ref{lm:exp_law}\ref{enum:exp_law:1}, we have the implication \ref{enum:exp_law:FT}$\Rightarrow$\ref{enum:exp_law:F}, and its inverse \ref{enum:exp_law:F}$\Rightarrow$\ref{enum:exp_law:FT} holds by Lemma~\ref{lm:exp_law}\ref{enum:exp_law:3} and local compactness of $\Mman$.

\ref{enum:exp_law:F}$\Leftrightarrow$\ref{enum:exp_law:loc_coords}.
Notice that continuity of $\hat{F}$, that is condition~\ref{enum:exp_law:F}, is equivalent to continuity of each of its local presentations.
Moreover, given a local presentation
\[
	(\Func_1,\ldots,\Func_n): \Lman \times \Vman \to \bR^n
\]
of $\Func$ one can define a local presentation of
\[
	\hat{\Func}:\Lman\times\Vman\to\bR^{n(1+m + m^2 +\cdots + m^r)}
\]
of $\hat{\Func}$ by
\[
\hat{\Func}(\ptl,\ptm) =
	\left(\ptm,
		\left\{
			\tfrac{\partial^i\Func_j(\ptl,\ptm)}{\partial\ptm_{a_1}\cdots\partial\ptm_{a_i}}
		\right\}_{a_1,\ldots,a_i=1,\ldots,m}^{ i=0,\ldots,\rx, \ j=1,\ldots,n} \right).
\]
Then $\hat{\Func}$ is continuous iff so is every of its coordinate functions being a partial derivatives up to order $\rx$ of coordinate functions of $\Func$, that is when condition~\ref{enum:exp_law:loc_coords} holds.
\end{proof}

%%%%%%% Characterization of C^{0,\infty} maps
\begin{corollary}[Characterization of $\Cr{0,\infty}$ maps]\label{cor:exp_law_inf}
Let $\Lman$ be a topological space, and $\func:\Lman\to\Ci{\Mman}{\Nman}$ be a map.
Then the following conditions are equivalent.
\begin{enumerate}[leftmargin=*, label={\rm(\arabic*)}]
\item\label{enum:exp_law:inf:f}
$\func:\Lman\to\Ci{\Mman}{\Nman}$ is continuous into the topology $\Wr{\infty}$ of $\Ci{\Mman}{\Nman}$.

\item\label{enum:exp_law:inf:FT}
The map $F:\Lman\times\Mman \to\Nman$, defined by $\Func(s,x) = \func(s)(x)$, is $\Cr{0,\infty}$.

\item\label{enum:exp_law:inf:loc_coords}
Let $\Func =( \Func_1,\ldots,\Func_n): \Lman\times\cone{m}{\Lambda} \ \supset \ \Lman\times \Vman \to \cone{n}{\Lambda'}$ be any local presentation of the map $\Func:\Lman\times\Mman \to\Nman$, $\Func(\ptl,\ptm) = \func(\ptl)(\ptm)$, where $\Vman\subset\cone{m}{\Lambda'}$ is an open set.
Then each partial derivative $\frac{\partial^i\Func_j}{\partial\ptm_{a_1}\cdots\partial\ptm_{a_i}}:\Lman \times \Vman \to \bR$ in $(\ptm_1,\ldots,\ptm_m)\in\Vman$ of each order of each coordinate function $\Func_i$ of $\Func$ is a continuous function on $\Lman \times \Vman$.
\qed
\end{enumerate}
\end{corollary}

\subsection*{Topology $\Wr{0,\rx}$.}
Let $\Lman$ be a topological space.
Notice that the map
\[
	\jetM{\rx}{\Mman}:  \Crm{0,\rx}{\Lman\times\Mman}{\Nman} 	\to  \Cont{\Lman\times\Mman}{\Jr{\rx}{\Mman}{\Nman}},
	\qquad
	\func \mapsto \jrm{\rx}{\Mman}{\func},
\]
is injective.
Endow $\Crm{0}{\Lman\times\Mman}{\Jr{\rx}{\Mman}{\Nman}}$ with topology $\Wr{0}$.
Then the induces topology on $\Crm{0,\rx}{\Lman\times\Mman}{\Nman}$ will be denoted by $\Wr{0,\rx}$.

\begin{lemma}\label{lm:exp_law_C0r}
Let $\Lman$ be a $T_3$ topological space and $0\leq \rx\leq \infty$.
Then the exponential map
$\EXP = \EXP_{\Lman,\Mman,\Nman}:\Crm{0,\rx}{\Lman\times\Mman}{\Nman} \to \Cont{\Lman}{\Crm{\rx}{\Mman}{\Nman}}$
\begin{equation}\label{equ:exp_law_C0r}
%	\EXP &= \EXP_{\Lman,\Mman,\Nman}:\Crm{0,\rx}{\Lman\times\Mman}{\Nman} \to \Cont{\Lman}{\Crm{\rx}{\Mman}{\Nman}},
	\EXP(\Func)(\ptl)(x) = \Func(\ptl,\ptm),
\end{equation}
is a homeomorphism.
\end{lemma}
\begin{proof}
Notice that we have the following commutative diagram:
\[
\xymatrix{
\ \Crm{0,\rx}{\Lman\times\Mman}{\Nman} \
	\ar@{^(->}[rr]^-{\jetM{\rx}{\Mman}}  \ar[d]_-{\EXP} &&
\ \Crm{0}{\Lman\times\Mman}{\Jr{\rx}{\Mman}{\Nman}} \
	\ar[d]^-{\EXP_{\Lman,\Mman,\Jr{\rx}{\Mman}{\Nman}}}_-{\cong} \\
%%%%%%
\ \Cont{\Lman}{\Crm{\rx}{\Mman}{\Nman}} \
     \ar@{^(->}[rr]^-{\indmap{(\ejr{\rx})}} &&
\ \Cont{\Lman}{\Cont{\Mman}{\Jr{\rx}{\Mman}{\Nman}}} \
}
\]
in which horizontal arrows are by definition topological embeddings, right arrow is a homeomorphism by Lemma~\ref{lm:exp_law}, and the left arrow (i.e.\! our map $\EXP$) is a bijection.
Hence it is a homeomorphism.
\end{proof}
\begin{corollary}\label{cor:exp_law_C0r_SLMN}
Let $\Lman$ be locally compact and $\Dman$ be a $T_3$ topological space.
Then we have the following homeomorphisms:
\begin{multline*}
\Cont{\Dman}{\Crm{0,\rx}{\Lman\times\Mman}{\Nman}}
\ \stackrel{\eqref{equ:exp_law_C0r}}{\cong} \
\Cont{\Dman}{\Cont{\Lman}{\Crm{\rx}{\Mman}{\Nman}}}
\ \stackrel{\eqref{equ:exp_law_general}}{\cong} \\ \cong
\Cont{\Dman\times\Lman}{\Crm{\rx}{\Mman}{\Nman}}
\ \stackrel{\eqref{equ:exp_law_C0r}}{\cong} \
\Crm{0,\rx}{(\Dman\times\Lman)\times\Mman}{\Nman}.
\qed
\end{multline*}
\end{corollary}

The following two lemmas are analogues of Lemma~\ref{lm:cotop}\ref{enum:lm:cotop::composition} for $\Cr{0,\rx}$ maps.
\begin{lemma}\label{lm:C0r_forward_composition}
\begin{enumerate}[leftmargin=*, label={\rm(\arabic*)}]
\item\label{enum:lm:C0r_forward_composition:a}
Let $\Func:\Lman \times \Mman \to \Nman$ be a $\Cr{0,\rx}$ map, $\Qman$ a smooth manifold, and $\phi:\Nman\to\Qman$ a $\Cr{\rx}$ map.
Then the composition $\phi\circ\Func:\Lman \times \Mman \to \Qman$ is $\Cr{0,\rx}$ as well.

\item\label{enum:lm:C0r_forward_composition:c}
Moreover, if $\Lman$ is a $T_3$ space, then the multiplication map
\[
\mult:
\Crm{0,\rx}{\Lman \times \Mman}{\Nman} \times \Crm{\rx}{\Nman}{\Qman}
\to
\Crm{0,\rx}{\Lman \times \Mman}{\Qman},
\qquad
\mult(\Func,\phi) = \phi\circ\Func,
\]
is continuous.
\end{enumerate}
\end{lemma}
\begin{proof}
\ref{enum:lm:C0r_forward_composition:a}
Since the property $\Cr{0,\rx}$ is local and says about coordinate functions of $\Func$, it suffices to consider the case when  $\Mman\subset\cone{m}{\Lambda}$, and $\Nman\subset\cone{n}{\Lambda'}$ are open subsets, and $\Qman=\bR$.

Let $\ptl\in \Lman$, $\ptm=(\ptm_1,\ldots,\ptm_k) \in \Mman$, $\ptn=(\ptn_1,\ldots,\ptn_k) \in \Nman$, and 
\[ \Func = (\Func_1,\ldots,\Func_n):\Lman \times \Mman \to \Nman\subset\bR^n\]
be the coordinate functions of $\Func$.
Then $\phi\circ\Func|_{\ptl\times\Mman}:\ptl\times\Mman \to \bR$ is $\Cr{\rx}$ which proves condition~\ref{enum:ft:ft_Cr} of Definition~\ref{def:part_Cr}.

Also, $\phi\circ\Func$ is continuous, and this proves~\ref{enum:ft:df_Cs} for $r=0$.
Moreover, if $r\geq1$, then
\begin{align}\label{equ:dphif_dxj}
\frac{\partial(\phi\circ\Func)}{\partial\ptm_j}(\ptl,\ptm) &=
\sum_{i=1}^{n} \frac{\partial\phi}{\partial\ptn_i}
	\bigl( \Func(\ptl,\ptm)\bigr)
	\frac{\partial\Func_i}{\partial\ptm_j}(\ptl,\ptm), \quad j=1,\ldots,m,
\end{align}
is continuous in $(\ptl,\ptm)$ due to continuity in $(\ptl,\ptm)$ of $\frac{\partial\phi}{\partial\ptn_i}$, $\frac{\partial\Func_i}{\partial\ptm_j}$, and $\Func$.
This proves~\ref{enum:ft:df_Cs} for $\rx=1$.

If $\rx>1$, then differentiating further~\eqref{equ:dphif_dxj} in $\ptm$ we will see that all partial derivatives of $\phi\circ\Func$ up to order $\rx$ are continuous in $(\ptl,\ptm)$ as well.

\medskip

\ref{enum:lm:C0r_forward_composition:c}
Now let us prove continuity of the map $\mult$.
By Lemma~\ref{lm:cotop}\ref{enum:lm:cotop::evaluation}, we have a continuous map $c:\Crm{\rx}{\Nman}{\Qman} \to \Cont{\Lman}{\Crm{\rx}{\Nman}{\Qman}}$ defined so that $c(\GFunc)$ is the constant map $\Lman \to \{\GFunc\}$ to the point $\GFunc\in\Crm{\rx}{\Nman}{\Qman}$.
Then $\mult$ coincides with the following composition of continuous maps:
\begin{multline*}
\Crm{0,\rx}{\Lman \times \Mman}{\Nman} \times \Crm{\rx}{\Nman}{\Qman}
	\ \xrightarrow[~\eqref{equ:exp_law_C0r}~]{\EXP_{\Lman,\Mman,\Nman} \ \times \ c} \\
%%%%%%%%%%%%%%%
\to \ \Cont{\Lman}{\Crm{\rx}{\Mman}{\Nman}} \times \Cont{\Lman}{\Crm{\rx}{\Nman}{\Qman}}
\ \xrightarrow[\text{Lemma~\ref{lm:cotop}\ref{enum:lm:cotop:map_into_product}}]{\cong} \\
%%%%%%%%%%%%%%%
\to \ \Cont{\Lman}{\Crm{\rx}{\Mman}{\Nman} \times \Crm{\rx}{\Nman}{\Qman}}
\ \xrightarrow{\text{Corollary~\ref{cor:Wr_cont_composition}}} \\
\to \ \Cont{\Lman}{\Crm{\rx}{\Mman}{\Qman}} \cong \Crm{0,\rx}{\Lman \times \Mman}{\Qman}.
\end{multline*}
Notice that $T_3$ property of $\Lman$ is used for continuity of the first arrow.
The second arrow is continuous due to Lemma~\ref{lm:cotop}\ref{enum:lm:cotop:map_into_product} and $T_3$ property of spaces of $\Cr{\rx}$ maps.
\end{proof}

\subsection*{$\Cr{0,\rx}$ morphisms of trivial fibrations}
Let 
\begin{align*}
&\pr{\Aman}:\Aman\times\Bman \to \Aman, &
&\pr{\Lman}:\Lman\times\Mman\to\Lman
\end{align*}
be trivial fibrations, i.e. projections to the corresponding first coordinates.
Then a \myemph{morphism} between these fibrations is a pair of continuous maps $\LMiso:\Aman\times\Bman \to \Lman\times\Mman$ and $\Hl:\Aman\to\Lman$ making commutative the following diagram:
\begin{equation}\label{equ:C0r_square_diag}
\xymatrix{
\Aman\times\Bman \ar[rr]^-{\LMiso} \ar[d]_-{\pr{\Aman}} && \Lman\times\Mman \ar[d]^-{\pr{\Lman}} \\
\Aman \ar[rr]^-{\Hl} && \Lman
}
\end{equation}
which just means that $\LMiso(a,b) = (\Hl(a), \Hm(a,b))$, where $\Hm:\Aman\times\Bman \to \Mman$ is a continuous map.
In this case we will write $\LMiso=(\Hl,\Hm)$.

If in addition $\Bman,\Mman$ are smooth manifolds and $\Hm\in\Crm{0,\rx}{\Aman\times\Bman}{\Mman}$ i.e. it is $\Cr{0,\rx}$ along $\Bman$, then $\LMiso$ will be called a \myemph{$\Cr{0,\rx}$ morphism}.

Thus the following \myemph{topological product}
\begin{align*}
    \Mor{0,\rx}{\Aman\times\Bman}{\Lman\times\Mman}&:=
    \Cont{\Aman}{\Lman} \times \Crm{0,\rx}{\Aman\times\Bman}{\Mman}
\end{align*}
can be regarded as the space of \myemph{$\Cr{0,\rx}$ morphisms} between trivial fibrations $\pr{\Aman}$ and $\pr{\Lman}$.

\begin{lemma}\label{lm:C0r_backward_composition}
Let $\Aman, \Lman$ be topological spaces and $\Bman,\Mman,\Nman$ be smooth manifolds.
\begin{enumerate}[leftmargin=*, label={\rm(\arabic*)}]
\item\label{enum:lm:C0r_backward_composition:a}
Let $\LMiso = (\Hl, \Hm) \in \Mor{0,\rx}{\Aman\times\Bman}{\Lman\times\Mman}$.
Then for any $\Cr{0,\rx}$ along $\Mman$ map $\Func:\Lman\times\Mman\to\Nman$ the composition $\Func\circ\LMiso:\Aman\times\Bman\to\Nman$ is $\Cr{0,\rx}$ along $\Bman$ as well.

\item\label{enum:lm:C0r_backward_composition:c}
If $\Aman$ and $\Lman$ are $T_3$ spaces, then the following composition map
\begin{gather*}
\mult: \ \Mor{0,\rx}{\Aman\times\Bman}{\Mman\times\Lman} \ \times \ \Crm{0,\rx}{\Lman\times\Mman}{\Nman} \ \to \
\Crm{0,\rx}{\Aman\times\Bman}{\Nman}, \\
\mult(\Hl, \Hm, \Func)(a,b) = \Func(\Hl(a), \Hm(a,b))
\end{gather*}
is continuous.
\end{enumerate}
\end{lemma}
\begin{proof}
\ref{enum:lm:C0r_backward_composition:a}
The proof is similar to Lemma~\ref{lm:C0r_forward_composition}.
Notice that $\Func\circ\LMiso$ is continuous and we get from~\eqref{equ:C0r_square_diag} that $\LMiso(a\times\Bman) \subset \Hl(a) \times \Mman$ for each $a\in\Aman$.
Hence the map
\[
    (\Func\circ\LMiso)_{a}:=\Func\circ\LMiso|_{a\times\Bman} = \bigl( \Func|_{\Hl(a)\times\Mman}\bigr) \circ \bigl( \LMiso|_{a\times\Bman} \bigr)
\]
is $\Cr{\rx}$.
Moreover, it is evident that the partial derivatives of order $\rx$ of $(\Func\circ\LMiso)_{a}$ along $\Bman$ in local coordinates are polynomials of partial derivatives of $\Func$ along $\Mman$ and $\LMiso$ along $\Bman$ of orders $\leq\rx$.
But the assumption that $\Func$ and $\LMiso$ are $\Cr{0,\rx}$ means that the latter derivatives are continuous on the corresponding products $\Aman\times\Bman$ and $\Lman\times\Mman$.
This implies continuity of the $\rx$-jet of $\Func\circ\LMiso$ along $\Bman$.
We leave the details for the reader.

\newcommand\arrsp{\qquad\qquad\qquad\qquad\qquad\qquad\qquad}
\newcommand\brrsp{\qquad}
\ref{enum:lm:C0r_backward_composition:c}
Notice that $\mult$ is the composition of the following continuous maps
\begin{align*}
   &\mult: \Mor{0,\rx}{\Aman\times\Bman}{\Mman\times\Lman} \ \times \ \Crm{0,\rx}{\Lman\times\Mman}{\Nman}  \equiv \\
%%%%%%%%%%%%
& \brrsp \equiv: \Cont{\Aman}{\Lman} \times \Crm{0,\rx}{\Aman\times\Bman}{\Mman} \times \Crm{0,\rx}{\Lman\times\Mman}{\Nman} \\
&\arrsp \xrightarrow[\eqref{equ:exp_law_C0r}]{(\phi, \ \eta, \ \Func) \ \mapsto \ (\,a\,\mapsto\,\eta_a, \ \phi, \ \ptl\,\mapsto\,\Func_{\ptl}\,)} \\
%%%%%%%%%%%%
& \brrsp 
\to \Cont{\Aman}{\Crm{\rx}{\Bman}{\Mman}} \times \Cont{\Aman}{\Lman} \times \Cont{\Lman}{\Crm{\rx}{\Mman}{\Nman}} \\
&\arrsp \xrightarrow[\eqref{equ:exp_law_general}]{(\,a\,\mapsto\,\eta_a, \ \phi, \ \ptl\,\mapsto\,\Func_{\ptl}\,) \ \mapsto \ (\,a\,\mapsto\,\eta_a, \  a\,\mapsto\,\Func_{\phi(a)}\,)} \\
% %%%%%%%%%%%%
& \brrsp
\to \Cont{\Aman}{\Crm{\rx}{\Bman}{\Mman}} \times \Cont{\Aman}{\Crm{\rx}{\Mman}{\Nman}} \\
&\arrsp 
\xrightarrow[\text{Lemma~\ref{lm:cotop}\ref{enum:lm:cotop:map_into_product}}]{(\,a\,\mapsto\,\eta_a, \  a\,\mapsto\,\Func_{\phi(a)}\,) \ \mapsto \ (\,a\,\mapsto\, (\eta_a, \Func_{\phi(a)})\,)} \\
% %%%%%%%%%%%%
& \brrsp 
\to \Cont{\Aman}{\Crm{\rx}{\Bman}{\Mman} \times \Crm{\rx}{\Mman}{\Nman}} \\
&\arrsp 
\xrightarrow[\text{Corollary~\ref{cor:Wr_cont_composition}}]{(\,a\,\mapsto\, (\eta_a, \Func_{\phi(a)})\,) \ \mapsto \ (a\,\mapsto\,\Func_{\phi(a)}\circ\eta)} \\
%%%%%%%%%%%%%%
& \brrsp
\to \Cont{\Aman}{\Crm{\rx}{\Bman}{\Nman}} \cong
\Crm{0,r}{\Aman\times\Bman}{\Nman}.
\qedhere
\end{align*}
\end{proof}

\subsection*{Description of the topology $\Wr{0,\rx}$ in terms of norms}
There is a well known description of Whitney topologies $\Wr{\rx}$ in terms of norms $\cnorm{\cdot}{\rx}{\Kman}$, see~\cite[\S2.2]{Hirsch:DiffTop}.
We will need a similar description of the topology $\Wr{0,\rx}$.

Let $\Lman$ be a topological space, $\Yman \subset \Lman \times \bR^m_{+}$ an open subset, $\Kman\subset\Yman$ a compact set, $r<\infty$, and $\Func=(\Func_1,\ldots,\Func_n):\Yman\to\cone{n}{\Lambda'}$ a map being $\Cr{0,\rx}$ along $\cone{m}{\Lambda}$.
Define the \myemph{$\Cr{0,\rx}$-norm of $\Func$ on $\Kman$} by the following formula:
\[
\cnorm{\Func}{0,\rx}{\Kman} :=
\sup_{
	\substack{ (\ptl,\ptm)\in\Kman, \ \ j=1,\ldots, n, \\[1mm]
	           i=0,\ldots,\rx, \ \  a_1,\ldots,a_i=0,\ldots,m}
}
\left| \frac{\partial^i\Func_j}{\partial\ptm_{a_1}\cdots\partial\ptm_{a_i}}(\ptl,\ptm) \right|,
\]
where $x=(\ptm_1,\ldots,\ptm_m)$.
It is an analogue of the usual $\Cr{\rx}$-norm $\cnorm{\Func}{\rx}{\Kman}$ but only in coordinates in $\bR^m$, see~\cite[\S2.2]{Hirsch:DiffTop}.

Let $\Uman\subset\Lman$ be an open subset, $\phi:\Vman\to\cone{m}{\Lambda}$ and $\psi:\Wman\to\cone{n}{\Lambda'}$ two charts from the atlases of $\Mman$ and $\Nman$ respectively, and $\Func:\Lman\times\Mman\to\Nman$ a $\Cr{0,\rx}$ along $\Mman$ map such that $\Func(\Uman\times\Vman) \subset \Wman$.
Then a \myemph{local presentation} of $\Func$ in these charts is the following map
\begin{gather*} %\label{equ:local_presentation_C0r}
    \Lman\times\cone{m}{\Lambda}
		\supset
	\Uman\times\phi(\Vman)
		\ \xrightarrow{~\hat{\Func}~} \
	\psi(\Wman) \subset  \cone{n}{\Lambda'},
	\\
	\hat{\Func}(\ptl,\ptm) = \psi \circ \Func(t, \phi^{-1}(x)).
\end{gather*}
Therefore if $\Kman\subset\Uman\times\Vman$ is a compact subset, then one can define the \myemph{$\Cr{0,\rx}$-norm $\relcnorm{\Func}{0,\rx}{\Kman}{\phi,\psi}$ of $\Func$ on $\Kman$ with respect to charts $(\phi,\psi)$} by
\[
	\relcnorm{\Func}{0,\rx}{\Kman}{\phi,\psi}:= \cnorm{\hat{\Func}}{0,\rx}{\Kman}.
\]
Given $\Func\in\Crm{0,\rx}{\Lman\times\Mman}{\Nman}$ such that $\Func(\Kman) \subset \Wman$ and $\eps>0$ denote
\[
 \NN(\Func, \Kman, \Wman, \phi,\psi, \eps) \!=\! \{
	G\in\Crm{0,\rx}{\Lman\times\Mman}{\Nman} \mid 
	G(\Kman) \subset \Wman,  \relcnorm{\Func- G}{0,\rx}{\Kman}{\phi,\psi} \!<\!\eps
 \}.
\]

The following statement is an analogue of definition of the topology $\Wr{\rx}$ from~\cite[\S2.1, Eq.(1)]{Hirsch:DiffTop} or~\cite[Prop.~8.2.8]{Mukherjee:DT:2015}.
The proof is almost literally the same as in~\cite[Prop.~8.2.8]{Mukherjee:DT:2015} and we leave it for the reader.
\begin{lemma}\label{lm:base_W0r}
The sets $\NN(\Func, \Kman, \Wman, \phi,\psi, \eps)$ constitute a subbase of the topology $\Wr{0,\rx}$ on the space $\Crm{0,\rx}{\Lman\times\Mman}{\Nman}$.
\qed
\end{lemma}

\begin{lemma}\label{lm:approx_by_convol}
{\rm(cf.\cite[Chapter~2, Theorem 2.3(c)]{Hirsch:DiffTop} and \cite[Lemma~10.1.13]{MargalefOuterelo:DT:1992} for manifolds with corners)}
Let $\Func: \cone{l}{\Lambda'} \times \cone{m}{\Lambda} \supset \Uman \to \cone{n}{\Lambda''}$ be a $\Cr{0,\rx}$ along $\cone{m}{\Lambda}$ map, $r<\infty$, and $\Kman\subset\Uman$ be a compact subset.
Then for each $\eps>0$ there exists a $\Cr{\infty}$ map $\GFunc: \Uman \to \cone{n}{\Lambda''}$ such that 
\[
\cnorm{\GFunc - \Func}{0,\rx}{\Kman} < \eps.
\]
Moreover, if $\Func(\partial\Uman) \subset \partial\cone{n}{\Lambda''}$, then one can also assume that $\GFunc(\partial\Uman) \subset \partial\cone{n}{\Lambda''}$ as well.
\end{lemma}
\begin{proof}
Again the proof almost literally repeats the proof of~\cite[Chapter~2, Theorem 2.3(c)]{Hirsch:DiffTop} for arbitrary maps, and~\cite[Chapter~2, Lemma~3.1]{Hirsch:DiffTop} (resp. \cite[Lemma~10.1.13]{MargalefOuterelo:DT:1992}) for maps between manifolds with boundary (resp. with corners) sending boundary to boundary.
The map $\GFunc$ should be taken in the form of convolution $\theta*\Func$ with some $\Cinfty$ function $\theta:\bR^l\times\bR^m\to[0,\infty)$ having a sufficiently small support and satisfying $\int_{\bR^l\times\bR^m} \theta = 1$.
Due to \cite[Chapter~2, Theorem 2.3(a),(b)]{Hirsch:DiffTop} the convolution uniformly approximates not only the initial map but each of partial derivatives of each of its coordinate functions, and this is enough for achieving the inequality $\cnorm{\GFunc - \Func}{0,\rx}{\Kman} < \eps$.
Also in the cases of sending boundary to boundary, one should choose $\theta$ in a special form of a product of functions depending on separate coordinates, \cite[Chapter~2, Lemma~3.1]{Hirsch:DiffTop} \cite[Lemma~10.1.13]{MargalefOuterelo:DT:1992}.
We leave the details for the reader.
\end{proof}

\begin{remark}\label{rem:Cerf_and_others}\rm
In the literature there are different approaches to the definition of the  ``right'' category for the maps belonging to the image of the embedding 
\[ \EXP_{\Lman,\Mman,\Nman}: \Crm{\rx}{\Lman\times\Mman}{\Nman} \subset \Cont{\Lman}{\Crm{\rx}{\Mman}{\Nman}} \]

For instance, in the paper~\cite[Definition~4.1.1]{Cerf:BSMF:1961} by J.~P.~Cerf the adjoint map $\EXP_{\Lman,\Mman,\Nman}(\Func)$ for $\Func\in \Crm{\rx}{\Lman\times\Mman}{\Nman}$ is called \myemph{$\rx$-admissible}%
\footnote{\,\myemph{$\rx$-permise} in French}.
Moreover, map $\phi:\mathcal{A}_1 \to \mathcal{A}_2$ between subsets $\mathcal{A}_i\subset \Crm{\rx}{\Mman_i}{\Nman_i}$ of spaces of $\Cr{\rx}$ maps between smooth manifolds is called \myemph{$\rx$-differentiable} whenever it induces a map $\Crm{\rx}{\Lman\times\Mman_1}{\Nman_1} \to \Crm{\rx}{\Lman\times\Mman_2}{\Nman_2}$, i.e. when
\[ 
\phi\bigl( \EXP_{\Lman,\Mman_1,\Nman_1}\bigl(\Crm{\rx}{\Lman\times\Mman_1}{\Nman_1}\bigr) \bigr) 
\subset 
\EXP_{\Lman,\Mman_2,\Nman_2}\bigl(\Crm{\rx}{\Lman\times\Mman_2}{\Nman_2}\bigr).
\]

Also, in~\cite{Poenaru:PMIHES:1970} the space $\Crm{\infty}{\Lman\times\Mman}{\Nman}$ is denoted by $\Crm{\infty}{\Lman}{ \Crm{\infty}{\Mman}{\Nman}}$.
\end{remark}

\section{Differentiable approximations}\label{sect:diff_approx}
We prove in Lemma~\ref{lm:approx_C0r} below a variant of the standard (relative) approximation theorem about density of $\Cr{\sx}$ maps in $\Wr{0,\rx}$ topologies for $\sx\geq \rx$.
We also establish existence of homotopies between the initial map and its approximation.
This will be the principal fact for proving weak homotopy equivalences.

\subsection*{Historical remarks}
It is shown in the textbook~\cite[Section 6.7]{Steenrod:FibreBundles:1999} by N.~Steenrod that every continuous section $\func:\Mman\to\Eman$ of a locally trivial fibration $p:\Eman\to\Mman$ between smooth manifolds can be arbitrary close approximated by a $\Cr{\sx}$ section $\func':\Mman\to\Eman$ for any $\sx>0$.
Moreover, if $\func$ is already $\Cr{\sx}$ on a \myemph{neighbourhood} $\Uman$ of a closed subset $\Xman\subset\Mman$, then one can assume that $\func'=\func$ on some smaller neighbourhood $\Vman\subset\Uman$ of $\Xman$.
Similar technique was then described in another classical textbook~\cite[Chapter 2, Theorem 2.5]{Hirsch:DiffTop} by M.~Hirsch for the proof that $\Crm{\sx}{\Mman}{\Nman}$ is dense in $\Crm{\rx}{\Mman}{\Nman}$ for $\rx<\sx$ with respect to either strong or weak corresponding topologies.
A generalization of that statement was also given in~\cite[Theorem~A.3.1]{Neeb:AIF:2002} for the case when $\Mman$ is a finite-dimensional $\sigma$-compact manifold and $\Nman$ is a locally convex space topological vector space. 

An additional observation which is not mentioned in textbooks~\cite{Steenrod:FibreBundles:1999, Hirsch:DiffTop} but is essential for our purposes is that \myemph{the approximation $\func'$ is homotopic to the initial map $\func$}.
Probably the reason of not taking that homotopy to account is that in general \myemph{a homotopy does not induce a continuous path into the space $\Crm{\rx}{\Mman}{\Eman}$ endowed with strong topology}. 
This kind of problem was tried to avoid by J.~P.~Cerf~\cite[Chapter~4]{Cerf:BSMF:1961} by considering only smooth homotopies, see Remark~\ref{rem:Cerf_and_others}.
For the continuity of that path one should require, e.g. that the homotopy is compactly supported, e.g.~\cite[Chapter 2, \S3, page 43]{GolubitskiGuillemin:StableMats:ENG:1973}.
These questions are usually considered for approximation of maps between ANR's in $\Cr{0}$ category, e.g.~\cite[Chapter IV, Theorem~1.1]{Hu:Retracts:1965}.
Existence of such homotopies is proved in the unpublished manuscript~\cite[Proposition~3]{KrieglMichor:homotopy:2002}, and in~\cite[Theorem~11]{Wockel:AMB:2009} under wide assumptions on $\Eman$. %, but again only with respect to strong $\Cr{0}$ topology.
It is also expressed in~\cite{Glockner:0812.4713:2008} in a more general context via special subsets $\mathrm{conv}_{2}(\Uman) = \{ tx+(1-t)y \mid x,y\in\Uman \}$ of convex hulls of open sets $\Uman$ in locally convex vector spaces, c.f.~\eqref{equ:gluing_map_q} below and Remark~\ref{rem:comparison_with_Glockner}.

As an application of existence of arbitrary small homotopies to a smooth approximation it is shown in~\cite[Corollary II.13]{MullerWockel:AG:2009} that if  $K$ is a Lie group modelled on a locally convex space and $M$ a finite-dimensional connected manifold with corners, then each continuous principal $K$-bundle over $M$ is continuously equivalent to a smooth principal bundle.
Moreover, two smooth principal $K$-bundles over $M$ are smoothly equivalent if and only if they are continuously equivalent.

Let us also mention that the technical assumption that \myemph{$\func$ must be $\Cr{\sx}$ on a neighbourhood of $\Xman$} is presented in all mentioned above papers and it due to gluing technique via partitions of unity.
It prevents to directly extend absolute case $\sAA \subset \sBB$ to the relative case $\AfX \subset \BfX$ and was the reason of introducing a notion of stabilizing pair of isotopies, see Remark~\ref{rem:non_ext_problem}.

\medskip 

In what follows let $0\leq \rx\leq \sx \leq \infty$, $\Lman,\Mman,\Nman$ be smooth manifolds, $\Xman\subset\Lman\times\Mman$ a closed subset, $\sBB \subset \Crm{0,\rx}{\Lman\times\Mman}{\Nman}$ an open subset, and $\sFunc\in\sBB$.
For a map $\pfunc\in\Crm{\sx}{\Lman\times\Mman}{\Nman}$ we also put 
\begin{align*}
\BfX &= \{ \gfunc \in \sBB \mid \gfunc|_{\Xman} = \pfunc|_{\Xman} \}, &
\AfX &= \BfX \,\cap\, \Crm{\sx}{\Lman\times\Mman}{\Nman}.
\end{align*}

\begin{definition}\label{def:good_homotopy}
A map $\BXSiso:[0,1]\times\Lman\times\Mman \to \Nman$ is a \myemph{\fgood\ homotopy} of $\sFunc$, if
\begin{enumerate}[label={\rm(\alph*)}]
\item\label{enum:good_homotopy:0}
$\BXSiso$ is $\Cr{0,\rx}$ along $\Mman$;

\item\label{enum:good_homotopy:1}
$\BXSiso_0 = \sFunc$;

\item\label{enum:good_homotopy:2}
for each $\ptt\in[0,1]$ the map $\BXSiso_{\ptt}:\Lman\times\Mman\to \Nman$, $\BXSiso_{\ptt}(\ptl,\ptm)=\BXSiso(\ptt,\ptl,\ptm)$, belongs to $\sBB$ and coincides with $\sFunc$ on $\Xman$;

\item\label{enum:good_homotopy:3}
if $\sFunc$ is $\Cr{\sx}$ on some submanifold $\Wman\subset\Lman\times\Mman$, then the restriction of $\BXSiso$ to $[0,1]\times\Wman$ is $\Cr{\sx}$ as well.
\end{enumerate}
\end{definition}

\begin{lemma}\label{lm:approx_C0r}{\rm(cf.~\cite[Chapter 2, Theorem 2.5]{Hirsch:DiffTop})}
\begin{enumerate}[leftmargin=*, label={\rm(\arabic*)}]
\item\label{enum:lm:approx_C0r:incl_cont}
The natural inclusion $\incl:\Crm{\sx}{\Lman\times\Mman}{\Nman} \subset \Crm{0,\rx}{\Lman\times\Mman}{\Nman}$ is continuous from the topology $\Wr{\sx}$ to the topology $\Wr{0,\rx}$.

\item\label{enum:lm:approx_C0r:dense}
The subset $\Crm{\sx}{\Lman\times\Mman}{\Nman}$ is dense in $\Crm{0,\rx}{\Lman\times\Mman}{\Nman}$ in the topology $\Wr{0,\rx}$.

\item\label{enum:lm:approx_C0r:appr}
Suppose $\sFunc\in\sBB$ is a map being $\Cr{\sx}$ on some neighbourhood of $\Xman$.
Then there exists a \fgood\ homotopy $\BXSiso:[0,1]\times\Lman\times\Mman \to \Nman$ of $\sFunc$ such $\BXSiso_1 \in \Crm{\sx}{\Lman\times\Mman}{\Nman}$.
\end{enumerate}
\end{lemma}
\begin{proof}
\ref{enum:lm:approx_C0r:incl_cont}
Notice that the inclusion $\incl:\Crm{\sx}{\Lman\times\Mman}{\Nman} \subset \Crm{0,\rx}{\Lman\times\Mman}{\Nman}$ is a composition of two upper arrows in the following commutative diagram:
\[
\xymatrix@C=1em{
\ \Crm{\sx}{\Lman\times\Mman}{\Nman} \
   \ar@/^5ex/[rr]^{\incl} \ar@{^(->}[r] &
\ \Crm{\rx}{\Lman\times\Mman}{\Nman} \
	\ar[r] \ar@{^(->}[d]_-{\ejr{\rx}} &
\ \Crm{0,\rx}{\Lman\times\Mman}{\Nman} \
    \ar@{^(->}[d]_-{\jetM{\rx}{\Mman}} \\
%%%%%%%%%%%
    &
\ \Crm{0}{\Lman\times\Mman}{\Jr{\rx}{\Lman\times\Mman}{\Nman}} \
	\ar[r]^-{\indmap{\xi}} &
\ \Crm{0}{\Lman\times\Mman}{\Jr{\rx}{\Mman}{\Nman}} \
}
\]
The first horizontal inclusion is continuous from topology $\Wr{\sx}$ to topology $\Wr{\rx}$ for $s\geq r$, and the second horizontal inclusion is continuous since the vertical arrows are topological embeddings and the lower arrow is continuous, see~\eqref{equ:jets_along_M}.

\ref{enum:lm:approx_C0r:dense}
The proof is similar to the standard approximation theorems, e.g.~\cite[Chapter 2, Theorems 2.4-2.6]{Hirsch:DiffTop} but instead of $\Cr{\rx}$ norms one should use $\Cr{0,\rx}$ norms and, in particular, Lemma~\ref{lm:approx_by_convol} which allows arbitrary close approximations of $\Cr{0,\rx}$ maps by $\Cinfty$ maps with respect to $\Cr{0,\rx}$ norms.
We leave the details for the reader.

\ref{enum:lm:approx_C0r:appr}
The proof is a combination of proofs of~\cite[Chapter IV, Theorem~1.1]{Hu:Retracts:1965} and~\cite[Theorem~2.5]{Hirsch:DiffTop}.
One can assume that $\Nman$ is a submanifold of $\bR^{\tau}$ for some large $\tau$, see~\cite[Proposition~3.3.9]{MargalefOuterelo:DT:1992} for manifolds with corners.
Let $\Vman$ be a tubular neighbourhood of $\Nman$ in $\bR^{\tau}$ and $\rho:\Vman\to\Nman$ a smooth retraction%
\footnote{\,If $\Nman$ is a manifold with corners, then existence of such an embedding $\Nman\subset\bR^{\tau}$ is shown in~\cite[Proposition~3.3.9]{MargalefOuterelo:DT:1992}, and existence of tubular neighbourhoods of $\Nman$ in $\bR^{\tau}$ is proved in~\cite[\S5, Th\'eor\`eme~1]{Douady:SHC:1961}, see also~\cite[Theorem~2.7]{AmmannIonescuNistor:DM:2006}.}.
Then by Lemma~\ref{lm:cotop}\ref{enum:lm:cotop::subbase_prop:open_subset}, $\Crm{0,\rx}{\Lman\times\Mman}{\Vman}$ can be regarded as an open subset of $\Crm{0,\rx}{\Lman\times\Mman}{\bR^{\tau}}$ with respect to the topology $\Wr{0,\rx}$.
Moreover, by Lemma~\ref{lm:C0r_forward_composition}, $\rho$ induces a continuous map
\begin{gather*}
\indmap{\rho}: \Crm{0,\rx}{\Lman\times\Mman}{\Vman} \to \Crm{0,\rx}{\Lman\times\Mman}{\Nman},
\qquad \vfunc \mapsto  \rho\circ \vfunc.
\end{gather*}

Let $\Uman$ be an open neighbourhood of $\Xman$ on which $\sFunc$ is $\Cr{\sx}$.
Fix any $\Cr{\sx}$ function $\lambda:\Lman\times\Mman\to[0,1]$ such that $\lambda=1$ on $\Xman$ and $\lambda=0$ on some neighbourhood of $\Mman\setminus\Uman$ and define the following map
\begin{gather*}
\eta: [0,1] \times \Crm{0,\rx}{\Lman\times\Mman}{\Nman} \to \Crm{0,\rx}{\Lman\times\Mman}{\bR^{\tau}}, \\
\eta(\ptt,a) = \sFunc + \ptt (1-\lambda) (a - \sFunc).
\end{gather*}

One easily checks that the addition and multiplication by scalars in $\Crm{0,\rx}{\Lman\times\Mman}{\bR^{\tau}}$ is continuous.
This implies that $\eta$ is continuous as well with respect to the corresponding topologies $\Wr{0,\rx}$.
Hence the following set $\mathcal{Y}:= \eta^{-1}\bigl( \Crm{0,\rx}{\Lman\times\Mman}{\Vman} \bigr)$ is open in the space $[0,1] \times\Crm{0,\rx}{\Lman\times\Mman}{\Nman}$, and we get a well-defined continuous map $q = \indmap{\rho}\circ \eta: \ \mathcal{Y} \to \Crm{0,\rx}{\Lman\times\Mman}{\Nman}$
\begin{equation}\label{equ:gluing_map_q}
%q = \indmap{\rho}\circ \eta: \ \mathcal{Y} \to \Crm{0,\rx}{\Lman\times\Mman}{\Nman},
%\qquad
q(\ptt,a) = \rho\circ\bigl(\sFunc + \ptt (1-\lambda) \cdot (a - \sFunc)\bigr).
\end{equation}

Evidently, $q([0,1]\times\{\sFunc\})=\{\sFunc\}$.
Hence $q^{-1}(\sBB)$ is an open neighbourhood of $[0,1]\times\{\sFunc\}$ in $\mathcal{Y}$ and therefore in $[0,1]\times \Crm{0,\rx}{\Lman\times\Mman}{\Nman}$.
Due to compactness of $[0,1]$, there exists a $\Wr{0,\rx}$-open neighbourhood $\WW \subset \sBB$ of $\sFunc$ such that $[0,1]\times\WW \subset q^{-1}(\sBB)$.

Now by~\ref{enum:lm:approx_C0r:dense} there exists a $\Cr{\sx}$ map $\omega\in\WW$.
Hence one can define the map $\BXSiso:[0,1]\times\Lman\times\Mman\to\Nman$ to be the adjoint map of the restriction of $q$ onto $[0,1]\times\{\omega\}$, that is
\[
	\BXSiso(\ptt,\ptl,\ptm) = q(\ptt,\omega)(\ptl,\ptm) =
	  \rho\circ\bigl(\sFunc + \ptt (1-\lambda) \cdot (\omega - \sFunc)\bigr)(\ptl,\ptm).
\]
We claim that $\BXSiso$ is the desired \fgood\ homotopy of $\sFunc$.
Indeed, since $\sFunc$ is $\Cr{0,\rx}$ along $\Mman$ and $\omega$ is $\Cr{\sx}$, the latter formula implies that
\begin{itemize}
\item
$\BXSiso$ is $\Cr{0,\rx}$ along $\Mman$ map,
\item
$\BXSiso_0=\sFunc$,
\item
$\BXSiso_{\ptt}=\sFunc$ on $\Xman$ and $\BXSiso_{\ptt} = q(\ptt,\omega) \in \sBB$ for all $\ptt\in[0,1]$,
\item
if $\sFunc$ is $\Cr{\sx}$ on some submanifold $\Wman\subset\Lman\times\Mman$, then the restriction of $\BXSiso$ to $[0,1]\times\Wman$ is $\Cr{\sx}$ as well,
\end{itemize}
i.e. all conditions of Definition~\ref{def:good_homotopy} hold.
Moreover, the latter map $\BXSiso_1$ is $\Cr{\sx}$.
\end{proof}

\begin{corollary}\label{cor:suff_cond_for_weh_AfX_BfX}
Let $\pfunc\in\Crm{\sx}{\Lman\times\Mman}{\Nman}$ and $\beta:(\Disk{k+1}, S^{k}) \to (\BfX, \AfX)$ be a continuous map of pairs such that its adjoint map
\[
	\BFunc:\Disk{k+1}\times\Lman\times\Mman\to\Nman,
	\qquad
	\BFunc(\pts,\ptl,\ptm) = \beta(\pts)(\ptl,\ptm),
\]
is $\Cr{\sx}$ on some neighbourhood of $\Disk{k+1} \times \Xman$.
Then $\beta$ is homotopic as a map of pairs to a map into $\AfX$.
\end{corollary}
\begin{proof}
Consider the following homeomorphisms being the compositions of the corresponding ``exponential laws'':
\begin{multline*}
 q: \Cont{\Disk{k+1}}{\Crm{0,\rx}{\Lman\times\Mman}{\Nman}}
	\ \stackrel{\eqref{equ:exp_law_C0r}}{\cong} \
	\Cont{\Disk{k+1}}{\Cont{\Lman}{\Crm{\rx}{\Mman}{\Nman}}}
	\ \stackrel{\eqref{equ:exp_law_general}}{\cong} \\ \cong
	\Cont{\Disk{k+1}\times\Lman}{\Crm{\rx}{\Mman}{\Nman}}
	\ \stackrel{\eqref{equ:exp_law_C0r}}{\cong} \
	\Crm{0,\rx}{(\Disk{k+1}\times\Lman)\times\Mman}{\Nman}.
\end{multline*}
Then the adjoint map $\BFunc$ to $\beta$ is the image of $q$, i.e. $\BFunc=q(\beta)$.

As $\sBB$ is open in $\Crm{0,\rx}{\Lman\times\Mman}{\Nman}$, we see that $\Cont{\Disk{k+1}}{\sBB}$ is open in $\Cont{\Disk{k+1}}{\Crm{0,\rx}{\Lman\times\Mman}{\Nman}}$, whence $\UNbh := q\bigl( \Cont{\Disk{k+1}}{\sBB} \bigr)$ is open in the space $\Crm{0,\rx}{(\Disk{k+1}\times\Lman)\times\Mman}{\Nman}$.
In particular, $\UNbh$ is an open neighbourhood of $\BFunc$.

Since by assumption $\BFunc$ is $\Cr{\sx}$ on some neighbourhood $\Disk{k+1}\times\Xman$, we get from Lemma~\ref{lm:approx_C0r}\ref{enum:lm:approx_C0r:appr} that there is a \betagood\ homotopy
\[G:[0,1]\times \Disk{k+1} \times \Lman \times\Mman \to \Nman\]
of $\BFunc$ such that $G_1$ is $\Cr{\sx}$.
In other words, $G$ has the following properties.
\begin{enumerate}[leftmargin=*, label={\rm(\alph*)}]
\item
$G$ is $\Cr{0,\rx}$ along $\Mman$.

\item
$G_0 = \BFunc$.

\item
For each $\ptt\in[0,1]$ the map $G_{\ptt}\in\UNbh$ and coincides with $\BFunc$ near $\Disk{k+1}\times\Xman$.

\item
If $\BFunc$ is $\Cr{\sx}$ on some submanifold $\Wman\subset\Disk{k+1}\times\Lman\times\Mman$, then the restriction of $G$ to $[0,1]\times\Wman$ is $\Cr{\sx}$ as well.
In particular, since for each $\pts\in S^{k}$ the restriction $\BFunc: \pts\times \Lman\times\Mman \to\Nman$ is $\Cr{\sx}$, we obtain that the restriction of $G$ to $[0,1]\times\pts\times \Lman\times\Mman$ is $\Cr{\sx}$ as well.

\item
$G_1\in \Crm{\sx}{\Disk{k+1}\times\Lman\times\Mman}{\Nman}$.
\end{enumerate}
Consider the adjoint map
\[
\gamma:	[0,1]\times\Disk{k+1} \to \Crm{0,\rx}{\Lman\times\Mman}{\Nman},
\qquad
\gamma(\ptt,\pts)(\ptl,\ptm) = G(\ptt,\pts,\ptl,\ptm).
\]
Then the above properties of $G$ mean that
\begin{enumerate}[label={\rm(\alph*)}]
\item
$\gamma$ is continuous;

\item
$\gamma(0) = \BFunc$;

\item
$\gamma([0,1]\times\Disk{k+1}) \subset \BfX$;

\item
$\gamma([0,1]\times S^k) \subset \AfX$;

\item
$\gamma(1 \times\Disk{k+1}) \subset \BfX \cap \Crm{\sx}{\Lman\times\Mman}{\Nman} = \AfX$.
\end{enumerate}
Hence $\gamma$ is the desired homotopy.
\end{proof}

%%%%%%%%%%%%%%%%%%%%%%%%%%%%%%%%%%%%%%%
\section{Proof of~\ref{enum:th:B_CsMN__B_rel_param:abs} and~\ref{enum:th:B_CsMN__B_rel_param:rel} of Theorem~\ref{th:B_CsMN__B_rel_param}}\label{sect:proof:th:B_CsMN__B_rel_param}
As in Section~\ref{sect:diff_approx} let $0\leq \rx\leq \sx \leq \infty$, $\Lman,\Mman,\Nman$ be smooth manifolds, $\Xman\subset\Lman\times\Mman$ a closed subset, 
$\sBB \subset \Crm{0,\rx}{\Lman\times\Mman}{\Nman}$ an open subset, and $\sAA = \sBB \,\cap\, \Crm{\sx}{\Lman\times\Mman}{\Nman}$.
For a mapping $\pfunc\in\Crm{\sx}{\Lman\times\Mman}{\Nman}$ we also denote 
$\BfX = \{ \gfunc \in \sBB \mid \gfunc|_{\Xman} = \pfunc|_{\Xman} \}$ and $\AfX = \BfX \,\cap\, \Crm{\sx}{\Lman\times\Mman}{\Nman}$.

First we reformulate Definition~\ref{def:stab_f_near_X_no_L} more explicitly for the context of Theorem~\ref{th:B_CsMN__B_rel_param}.
\begin{definition}\label{def:stab_f_near_X_U}
We will say that a pair of $\Cr{\sx}$ isotopies 
\begin{align}\label{equ:stab_iso}
	&\LMiso=(\Hl,\Hm):[0,1]\times\Lman\times\Mman \to \Lman\times\Mman, &
	&\Niso:[0,1]\times\Nman \to \Nman
\end{align}
\myemph{leafwise stabilizes $\pfunc$ near $\Xman$} whenever the following conditions hold:
\begin{enumerate}[label={\rm(\alph*)}]
\item\label{enum:def:stab_f_near_X_U:H_morphism}
for each $\ptt\in[0,1]$ the first coordinate function $\lambda:[0,1]\times\Lman\times\Mman\to\Lman$ of $\LMiso$ does not depend on $\ptm\in\Mman$, so each $\LMiso_{\ptt}$ is an automorphism of the trivial fibration $\Lman\times\Mman \to \Lman$;

\item\label{enum:def:stab_f_near_X_U:H0_Phi0}
$\LMiso_0 = \id_{\Lman\times\Mman}$ and $\Niso_0 = \id_{\Nman}$;

\item\label{enum:def:stab_f_near_X_U:HtX_IntX}
$\LMiso_{\ptt}(\Xman) \subset \Int{\Xman}$ for $0 < \ptt \leq 1$;

\item\label{enum:def:stab_f_near_X_U:Pt_f_Ht__f}
$\Niso_{\ptt} \circ \pfunc \circ \LMiso_{\ptt} = \pfunc$ on some neighbourhood of $\Xman$ (which may depend on $\ptt$) for each $\ptt\in(0,1]$.
\end{enumerate}
\end{definition}

The following theorem is just a reformulation of the first two statements of Theorem~\ref{th:B_CsMN__B_rel_param} in terms of $\Cr{0,\rx}$ maps. 
\begin{theorem}\label{th:B_CsMN__B_rel_param_U}
\begin{enumerate}[leftmargin=*, label={\rm(\arabic*)}]
\item\label{enum:th:B_CsMN__B_rel_param_U:abs}
The inclusion $\incl: \sAA \subset \sBB$ is a weak homotopy equivalence.

\item\label{enum:th:B_CsMN__B_rel_param_U:rel}
Suppose there exists a pair of $\Cr{\sx}$ isotopies~\eqref{equ:stab_iso} which leafwise stabilizes $\pfunc$ near $\Xman$.
Then the inclusion $\incl: \AfX \subset \BfX$ is a weak homotopy equivalence as well.
\end{enumerate}
\end{theorem}
 \begin{proof}
Consider the following conditions.
\begin{enumerate}[label={\rm(\roman*)}]
\item\label{enum:th:B_CsMN__B_rel_param_U:leaf_wise} 
There exists a pair of $\Cr{\sx}$ isotopies leafwise stabilizing $\pfunc$ near $\Xman$.

\item\label{enum:th:B_CsMN__B_rel_param_U:make_Cs} 
Every continuous map $\beta:(\Disk{k+1}, S^{k}) \to (\BfX, \AfX)$ is homotopic as a map of pairs to a map $\hat{\beta}$ whose adjoint mapping
\[
    \hat{\BFunc}:\Disk{k+1}\times\Lman\times\Mman\to\Nman, \qquad
    \hat{\BFunc}(\pts,\ptl,\ptm) = \beta(\pts)(\ptl,\ptm),
\]
is $\Cr{\sx}$ on \myemph{some neighbourhood} of $\Disk{k+1}\times\Xman$.

\item\label{enum:th:B_CsMN__B_rel_param_U:AfX_BfX_whe}
The inclusion $\incl: \AfX \subset \BfX$ is a weak homotopy equivalence.
\end{enumerate}

We claim that~\ref{enum:th:B_CsMN__B_rel_param_U:leaf_wise}$\Rightarrow$\ref{enum:th:B_CsMN__B_rel_param_U:make_Cs}$\Rightarrow$\ref{enum:th:B_CsMN__B_rel_param_U:AfX_BfX_whe}.
This will imply statement~\ref{enum:th:B_CsMN__B_rel_param_U:rel}.

Also, if $\Xman =\varnothing$, then condition~\ref{enum:th:B_CsMN__B_rel_param_U:make_Cs} trivially holds since $\Disk{k+1}\times\Xman=\varnothing$.
Therefore~\ref{enum:th:B_CsMN__B_rel_param_U:AfX_BfX_whe} will also hold, that is the inclusion $\incl:\sAA \equiv \AfX \subset \BfX \equiv \sBB$ will be a weak homotopy equivalence, and thus the implication \ref{enum:th:B_CsMN__B_rel_param_U:make_Cs}$\Rightarrow$\ref{enum:th:B_CsMN__B_rel_param_U:AfX_BfX_whe} gives a proof of~\ref{enum:th:B_CsMN__B_rel_param_U:abs}. 

\medskip

\ref{enum:th:B_CsMN__B_rel_param_U:make_Cs}$\Rightarrow$\ref{enum:th:B_CsMN__B_rel_param_U:AfX_BfX_whe}
By Lemma~\ref{lm:char_whe}\ref{enum:lm:char_whe:adm_maps_pairs} it suffices to show that each $\Wr{s}$-admissible map $\beta:(\Disk{k+1},S^{k}) \to (\BfX, \AfX)$ is homotopic via a $\Wr{s}$-admissible homotopy of pairs to a map into $\AfX$.

Condition~\ref{enum:th:B_CsMN__B_rel_param_U:make_Cs} says that such a map is homotopic as a map of pairs $(\Disk{k+1}, S^{k}) \to (\BfX, \AfX)$ to a map $\hat{\beta}$ whose adjoint mapping is $\Cr{\sx}$ near $\Disk{k+1}\times\Xman$.
Then by Corollary~\ref{cor:suff_cond_for_weh_AfX_BfX}, $\hat{\beta}$ is homotopic as a map of pairs to a map into $\AfX$.

\medskip 

\ref{enum:th:B_CsMN__B_rel_param_U:leaf_wise}$\Rightarrow$\ref{enum:th:B_CsMN__B_rel_param_U:make_Cs}
Let $\BFunc:\Disk{k+1}\times\Lman\times\Mman\to\Nman$ be the adjoint map of $\beta$, so 
\begin{equation}\label{equ:B_phi_on_X}
	\BFunc(\pts,\ptl,\ptm) = \beta(\pts)(\ptl,\ptm)= \pfunc(\ptl,\ptm) \ \text{for} \ (\pts,\ptl,\ptm)\in\Disk{k+1}\times\Xman.
\end{equation}

Let also~\eqref{equ:stab_iso} be two isotopies which leafwise stabilize $\pfunc$ near $\Xman$.
Then it easily follows from Lemmas~\ref{lm:C0r_forward_composition} and~\ref{lm:C0r_backward_composition} that the following map
\begin{gather*}
	\gamma:[0,1] \to \Crm{0,\rx}{(\Disk{k+1}\times\Lman)\times\Mman}{\Nman}, \\
%%%%%%%%%%%%%%%%%
	\gamma(\ptt)(\pts,\ptl,\ptm) = 
	  \Niso_{\ptt} \circ \BFunc_{\pts} \circ \LMiso_{\ptt}(\ptl,\ptm) = 
	  \Niso\bigl( \ptt, \BFunc\bigl(\pts, \LMiso_{\ptt}(\ptl,\ptm) \bigr)\bigr)
\end{gather*}
is \myemph{well-defined} and \myemph{continuous}.

\begin{sublemma}\label{lm:stab_near_X_prop}
For each $0<\ptt\leq 1$ there exists an open neighbourhood $\Vman_{\ptt}$ of $\Xman$ in $\Lman\times\Mman$, such that 
\begin{equation}
\gamma(\ptt)(\pts,\ptl,\ptm) = \pfunc(\ptl,\ptm), \qquad (\pts,\ptl,\ptm)\in \Disk{k+1}\times\Vman_{\ptt}.
\end{equation}
In particular, $\gamma(\ptt)$ is  $\Cr{\sx}$ near $\Disk{k+1}\times\Xman$.
\end{sublemma}
\begin{proof}
By assumption~\ref{enum:def:stab_f_near_X_U:Pt_f_Ht__f} of Definition~\ref{def:stab_f_near_X_U}, $\Niso_{\ptt} \circ \pfunc \circ \LMiso_{\ptt} = \pfunc$ on some neighbourhood $\Uman_{\ptt}$ of $\Xman$ for each $\ptt\in(0,1]$.
As $\LMiso_{\ptt}(\Xman) \subset \Int{\Xman}$ for $0 < \ptt \leq 1$, it also follows from assumption~\ref{enum:def:stab_f_near_X_U:HtX_IntX} that $\Xman\subset\Int{\bigl(\LMiso^{-1}_{\ptt}(\Xman)\bigr)}$.
We claim that $\Vman_{\ptt}:= \Uman_{\ptt} \cap \Int{\bigl(\LMiso^{-1}_{\ptt}(\Xman)\bigr)}$ satisfies the statement of our lemma.

Indeed, if $(\pts,\ptl,\ptm) \in\Disk{k+1}\times \Int{\LMiso^{-1}_{\ptt}(\Xman)}$, then $ \LMiso_{\ptt}(\ptl,\ptm) \in \Xman$, and therefore by~\eqref{equ:B_phi_on_X} 
\[ 
	\BFunc\bigl( \pts, \LMiso_{\ptt}(\ptl,\ptm)\bigr) =   \BFunc_{\pts} \circ \LMiso_{\ptt}(\ptl,\ptm) = \pfunc \circ \LMiso_{\ptt}(\ptl,\ptm).
\]
Hence if $\pts\in\Disk{k+1}$ and $(\ptl,\ptm) \in\Vman_{\ptt} = \Uman_{\ptt}\cap\Int{\bigl(\LMiso^{-1}_{\ptt}(\Xman)\bigr)}$, then 
\[
\gamma(\ptt)(\pts,\ptl,\ptm)  = \Niso_{\ptt} \circ \BFunc_{\pts} \circ \LMiso_{\ptt}(\ptl,\ptm) =
 \Niso_{\ptt} \circ \pfunc \circ \LMiso_{\ptt}(\ptl,\ptm) = 
 \pfunc(\ptl,\ptm).
\qedhere
\]
\end{proof}
Consider another (continuous due to Corollary~\ref{cor:exp_law_C0r_SLMN}) adjoint map to $\gamma$
\begin{align*}
&\GFunc: [0,\eps]\times \Disk{k+1} \to \Crm{0,\rx}{\Lman\times\Mman}{\Nman}, &
&\GFunc(\ptt,\pts)(\ptl,\ptm) = \Niso_{\ptt} \circ \BFunc \circ \LMiso_{\ptt}(\pts,\ptl,\ptm),
\end{align*}
being a homotopy of $\beta$.
As shown above, $\GFunc([0,\eps]\times \Disk{k+1}) \subset \BfX$.
Moreover, since $\beta(S^{k})$ consists of $\Cr{\sx}$ maps and $\LMiso$ and $\Niso$ are also $\Cr{\sx}$, it follows that $\GFunc([0,\eps]\times S^{k}) \subset \AfX$.
Hence $\GFunc$ is the desired homotopy of $\beta$ to a map $\gamma(\eps)$ such that its adjoint map $\GFunc_{\eps}$ is $\Cr{\sx}$ near $\Xman$.

Theorem~\ref{th:B_CsMN__B_rel_param_U} is completed.
\end{proof}

\begin{remark}\label{rem:non_ext_problem}\rm
Let $\beta:(\Disk{k+1}, S^{k}) \to (\BfX, \AfX)$ be a continuous map and $\BFunc:\Disk{k+1}\times\Lman\times\Mman\to\Nman$ be its adjoint map.
Then the restriction of $\BFunc$ to $\Disk{k+1}\times\Xman$ is given by
\[
    \BFunc(\pts,\ptl,\ptm)=\func(\ptl,\ptm).
\]
In particular, if $\Xman$ is a submanifold of $\Lman\times\Mman$, then $\Disk{k+1}\times\Xman$ is a submanifold of $\Disk{k+1}\times\Lman\times\Mman$ and $\BFunc$ is $\Cr{\sx}$ on $\Disk{k+1}\times\Xman$.
However one can not guarantee that $\BFunc$ is $\Cr{\sx}$ \myemph{on a neighbourhood} of $\Disk{k+1}\times\Xman$, and therefore can not immediately apply Lemma~\ref{lm:approx_C0r}\ref{enum:lm:approx_C0r:appr} to approximate $\BFunc$ by a $\Cr{\sx}$ map of all of $\Lman\times\Mman$ which coincides with $\BFunc$ on $\Disk{k+1}\times\Xman$.

Moreover, by Whitney extension theorem we can extend $\BFunc|_{\Disk{k+1}\times\Xman}$ to a $\Cr{\sx}$ map, and by previous technique extend it further to a global map $\hat{\BFunc}:\Disk{k+1}\times\Lman\times\Mman\to\Nman$, and even find a homotopy between $\BFunc$ and $\hat{\BFunc}$ relatively to $\Disk{k+1}\times\Xman$.
However, one can not guarantee that $\hat{\BFunc}$ and all that homotopy is contained in $\sBB$.
\end{remark}

\begin{example}
Let $\sBB = \{ \BFunc\in\Crm{1}{\bR}{\bR} \mid \func'(0)>0\}$, $\Xman = \{0\} \subset \bR$, and $\BFunc\in\sBB$ be given by $\BFunc(\ptm)=\ptm$.
Notice that $\BFunc|_{\Xman}: \Xman \to \bR$ is the zero function, and it extends for example to a function $\hat{\BFunc}\in\Crm{\infty}{\bR}{\bR}$ given by $\hat{\BFunc}(x) = -x$.
Then $\hat{\BFunc}$ does not belong to $\sBB$ and can not be connected with $\BFunc$ by a continuous path in $\sBB$.

This illustrates that the problem whether $\incl:\AfX \subset \BfX$ is a weak homotopy equivalence might essentially depend on $\func$, $\sBB$, $\Xman$, $\rx$ and $\sx$, c.f. \myemph{Lojasiewicz inequalities} in~\cite{Lojasiewicz:SM:1959, Malgrange:Ideals:1967}, and \myemph{no narrow fjord condition} in~\cite{RobertsSchmeding:1801.04126}.
\end{example}

\begin{remark}[Comparison with results of~\cite{Glockner:0812.4713:2008}]
\label{rem:comparison_with_Glockner}
\rm
As mentioned above, the statement that the inclusion $\incl:\Crm{\sx}{\Mman}{\Nman} \subset \Crm{\rx}{\Mman}{\Nman}$ is a weak homotopy equivalence follows from a very general~\cite[Proposition~6.1]{Glockner:0812.4713:2008}.

Let $\Qman$ be a topological space and $(\Qman_{\alpha})_{\alpha\in A}$ a directed family of subsets such that $\Qman_{\infty} := \cup_{\alpha\in A}\Qman_{\alpha}$ is dense in $\Qman$.
Endow each $\Qman_{\alpha}$ and $\Qman_{\infty}$ with some topologies such that all the inclusions $\Qman_{\alpha} \to \Qman_{\beta}$ for all $\alpha\leq\beta$, $\Qman_{\alpha} \to \Qman_{\infty}$, and $\Qman_{\infty} \to \Qman$ are continuous.
Then~\cite[Proposition~6.1]{Glockner:0812.4713:2008} claims that the inclusion $\Qman_{\infty}\subset\Qman$ is a weak homotopy equivalence, whenever 
$\Qman$ can be covered by \myemph{cores of well-filled charts} and $\Qman_{\infty}$ is \myemph{compactly retractive}.

We refer the reader for details and precise definitions to~\cite{Glockner:0812.4713:2008} however let us mention that in our case $\Qman=\Crm{\rx}{\Mman}{\Nman}$, $A=\{\alpha\}$ consists of a unique element, so $\Qman_{\alpha} = \Qman_{\infty} = \Crm{\sx}{\Mman}{\Nman}$, and property of being \myemph{compactly retractive}
automatically holds.
Moreover, a \myemph{well-filled} chart in $\Qman$ is a homeomorphism $\Qman \supset \Uman \xrightarrow{\eta} \Vman \subset \Eman$ of an open subset $\Uman$ of $\Qman$ onto an open subset $\Vman$ of a locally convex topological vector space $\Eman$ such that there exists an open subset $\Vman'\subset\Vman$ satisfying
\begin{enumerate}[label={\rm(\roman*)}]
\item\label{enum:cond:core:0} $t\func+(1-t)\gfunc\in \Vman'$ for all $\func,\gfunc\in\Vman'$, and
\item\label{enum:cond:core:1} $t\func+(1-t)\gfunc\in \Vman'_{\infty}:= \Vman' \cap \eta(\Uman\cap \Qman_{\infty})$ for all $\func,\gfunc\in\Vman'_{\infty}$.
\end{enumerate}
Then $\Uman' = \eta^{-1}(\Vman')$ is called a \myemph{core} of the well-filled chart $\eta$.

Thus by~\ref{enum:cond:core:0} any $\Cr{\rx}$ maps $\phi, \psi$ in the core $\Uman'$ are ``canonically'' homotopic in that core, in particular in the class of $\Cr{\rx}$ maps.
Moreover, due to~\ref{enum:cond:core:1}, if $\phi, \psi$ are $\Cr{\sx}$, then that homotopy consists of $\Cr{\sx}$ maps.
One can easily check that if $\rx\geq1$ then these properties hold for the inclusion $\Crm{\sx}{\Mman}{\Nman} \subset \Crm{\rx}{\Mman}{\Nman}$, since a neighbourhood of each $\func\in \Crm{\rx}{\Mman}{\Nman}$ can be identified with the space of sections of certain vector bundle over $\Mman$.
Then~\cite[Proposition~6.1]{Glockner:0812.4713:2008} implies that the inclusion is a weak homotopy equivalence.

Property~\ref{enum:cond:core:1} play in the proof of~\cite[Proposition~6.1]{Glockner:0812.4713:2008} the same role as condition~\ref{enum:good_homotopy:3} of Definition~\ref{def:good_homotopy} of \fgood\ homotopy (which is guaranteed by formula~\eqref{equ:gluing_map_q}), and expresses the same property as the last sentence of~\cite[Lemma~II.4]{MullerWockel:AG:2009}.
\end{remark}

\section{Maps of manifolds sending boundary to boundary}\label{sect:maps_of_pairs}
For smooth manifolds $\Mman,\Nman$, their subsets $\Yman\subset\Mman$ and $\Zman\subset\Nman$, a topological space $\Lman$, and a subset $\Xman\subset \Lman\times\Mman$ put
\begin{gather*}
	\CrmPairs{\rx}{\Mman}{\Yman}{\Nman}{\Zman} \, := \,
	\{ \func\in\Crm{\rx}{\Mman}{\Nman} \mid \func(\Yman) \subset \Zman\}, \\ %[0.5ex]
	%%%%%%%%%%%%%%%%%%
	\CrmPairs{0,\rx}{\Lman\times\Mman}{\Xman}{\Nman}{\Zman}  \, := \,
	\{ \Func\in\Crm{0,\rx}{\Lman\times\Mman}{\Nman} \mid \Func(\Xman) \subset \Zman \}.
\end{gather*}
If $\Lman$ is a $T_3$ space, then the homeomorphism~\eqref{equ:exp_law_C0r} of Lemma~\ref{lm:exp_law_C0r} induces a homeomorphism
\[
\EXP_{\Lman,\Mman,\Nman}: \,
\CrmLMdMNdN{0,\rx} \ \cong \ \Cont{\Lman}{\CrMdMNdN{\rx}}.
\]

Notice that if $\partial\Mman\not=\varnothing$, then the group $\Diff^{\rx}(\Mman)$ of diffeomorphisms of $\Mman$ is not open in $\Crm{\rx}{\Mman}{\Mman}$ though it is open in the subset $\CrMdMMdM{\rx}$ consisting of self maps of pairs $\func:(\Mman,\partial\Mman)\to(\Mman,\partial\Mman)$.
Therefore the first two statements of Theorem~\ref{th:B_CsMN__B_rel_param} are not applicable for the proof that for $0<\rx < \sx \leq\infty$ the inclusion $\Diff^{\sx}(\Mman) \subset \Diff^{\rx}(\Mman)$ is a weak homotopy equivalence, and it would be desirable to have such type of results for open subsets of $\CrMdMNdN{\rx}$.

We will briefly discuss how to extend the technique developed in the previous sections to the situation of maps sending boundary to boundary.

The following statement shows that Theorem~\ref{th:B_CsMN__B_rel_param_U} holds if we replace 
\begin{itemize}
\item $\Crm{0,\rx}{\Lman\times\Mman}{\Nman}$ with $\CrmLMdMNdN{0,\rx}$ and
\item $\Crm{\sx}{\Lman\times\Mman}{\Nman}$ with $\CrmLMdMNdN{\sx}$
\end{itemize}
We will just indicate the principal changes.

Let $0\leq\rx\leq\sx\leq\infty$, $\Lman$ be a smooth manifold possibly with corners, $\Xman\subset\Lman\times\Mman$ a closed subset, $\sBB \subset \CrmLMdMNdN{0,\rx}$ an open subset, and $\sAA = \sBB \,\cap\, \Crm{\sx}{\Lman\times\Mman}{\Nman}$.
For a map $\pfunc\in\CrmLMdMNdN{\sx}$ we also denote as before
\begin{align*}
	\BfX &= \{ \gfunc \in \sBB \mid \gfunc|_{\Xman} = \pfunc|_{\Xman} \}, &
	\AfX &= \BfX \,\cap\, \Crm{\sx}{\Lman\times\Mman}{\Nman}.
\end{align*}
\begin{theorem}\label{th:B_CsMN__B_rel_param_pairs}
\begin{enumerate}[leftmargin=*, label={\rm(\arabic*)}]
\item\label{enum:th:B_CsMN__B_rel_param_pairs:dense}
$\CrmLMdMNdN{\sx}$ is dense in $\CrmLMdMNdN{0,\rx}$ in the topology $\Wr{0,\rx}$.

\item\label{enum:th:B_CsMN__B_rel_param_pairs:B}
Let $\sFunc\in\sBB$ be a map which is $\Cr{\sx}$ on some neighbourhood of $\Xman$.
Then there exists a \fgood\ homotopy $\BXSiso:[0,1]\times\Lman\times\Mman \to \Nman$ of $\sFunc$ such that 
%$\BXSiso_1 \in \CrmLMdMNdN{\sx}$ and $\BXSiso\bigl([0,1]\times\Lman\times\partial\Mman\bigr) \subset \partial\Nman$, i.e. 
$\BXSiso_{\ptt}\in\sBB$ for all $\ptt\in[0,1]$, and $\BXSiso_1 \in \sAA$.

\item\label{enum:th:B_CsMN__B_rel_param_pairs:abs}
The inclusion $\incl: \sAA \subset \sBB$ is a weak homotopy equivalence.

\item\label{enum:th:B_CsMN__B_rel_param_pairs:rel}
Suppose there exists a pair of $\Cr{\sx}$ isotopies~\eqref{equ:stab_iso} which leafwise stabilizes $\pfunc$ near $\Xman$.
Then the inclusion $\incl: \AfX \subset \BfX$ is a weak homotopy equivalence as well.
\end{enumerate}
\end{theorem}
\begin{proof}[Sketch of proof]
\ref{enum:th:B_CsMN__B_rel_param_pairs:dense}
The proof follows the line of approximation theorems for maps of sending boundary to boundary by using Lemma~\ref{lm:approx_by_convol}.
%, \cite[Chapter 2, Lemma~3.1]{Hirsch:DiffTop}, and~\cite[Lemma~10.1.13]{MargalefOuterelo:DT:1992}.

\ref{enum:th:B_CsMN__B_rel_param_pairs:B}
This is a variant of the statement~\ref{enum:lm:approx_C0r:incl_cont} of Lemma~\ref{lm:approx_C0r}.
The proof is literally the same, however one only needs to choose an embedding $\Nman\subset\bR^{\tau}$, a tubular neighbourhood $\Vman$ of $\Nman$, and a retraction $\rho:\Vman\to\Nman$ so that $\partial\Nman = \Nman \cap \bR^{\tau-1}$ and $\rho^{-1}(\partial\Nman) \subset \Vman\cap\bR^{\tau-1}$.
Those conditions can easily be satisfied.
Then the formula~\eqref{equ:gluing_map_q} gives a well-defined map 
\[ 
q: \mathcal{Y}=\eta^{-1}\bigl( 
\CrmPairs{0,\rx}{\Lman\times\Mman}{\Lman\times\partial\Mman}{\Vman}{\Vman\cap\bR^{\tau-1}}
% \Cdrm{0,\rx}{\Lman\times\Mman}{\Vman} 
 \bigr) \to \CrmLMdMNdN{0,\rx}
\]
which implies existence of the desired homotopy $\BXSiso$.

The proof of statements~\ref{enum:th:B_CsMN__B_rel_param_pairs:abs} and~\ref{enum:th:B_CsMN__B_rel_param_pairs:rel} also literally repeats the arguments of Corollary~\ref{cor:suff_cond_for_weh_AfX_BfX} and Theorem~\ref{th:B_CsMN__B_rel_param_U}. 
We leave the details for the reader.
\end{proof}

In particular, statements~\ref{enum:th:B_CsMN__B_rel_param_pairs:abs} and~\ref{enum:th:B_CsMN__B_rel_param_pairs:rel} of Theorem~\ref{th:B_CsMN__B_rel_param_pairs} prove statement~\ref{enum:th:B_CsMN__B_rel_param:pairs} of Theorem~\ref{th:B_CsMN__B_rel_param}.

\subsection*{Acknowledgment}
The authors are sincerely indebted to Professor Helge Gl\"{o}ckner for informing us about papers~\cite{Glockner:0812.4713:2008, Wockel:AMB:2009, Glockner:1811.02888:2018} from which and from references there they (the authors) become aware about current state of the discussed problem.

% \bibliographystyle{pigc_plain}
% \bibliography{biblio}

\def\cprime{$'$}

\end{document}